\documentclass[english,11pt]{article}
\usepackage{geometry}
\usepackage{amssymb}
\usepackage{amsthm}
\usepackage{amsmath}
\usepackage{enumitem}   
\usepackage[colorlinks=true]{hyperref}

\usepackage{tikz-cd}
\usepackage{subfig}

\theoremstyle{plain}
\newtheorem{theorem}{Theorem}[section]
\newtheorem{proposition}{Proposition}[section]
\newtheorem{lemma}{Lemma}[section]
\newtheorem{corollary}{Corollary}[section]

\theoremstyle{remark}

\newtheorem{remark}{Remark}[section]

\newtheorem{claim}{Claim}[section]

\theoremstyle{definition}
\newtheorem{definition}{Definition}[section]
\newtheorem{example}{Example}[section]

\newcommand{\id}{\mathrm{id}}
\renewcommand{\int}[1]{\mathrm{int}(#1)}
\newcommand{\odd}[1]{\partial_\mathrm{odd}(#1)}
\newcommand{\oddn}[2]{\partial^{#1}_\mathrm{odd}(#2)}
\newcommand{\bd}[1]{\mathrm{bd}(#1)}
\newcommand{\op}[1]{\mathrm{op}(#1)}
\newcommand{\Ne}{\mathcal{N}}

\newcommand{\Sp}{\mathbb{S}}
\newcommand{\BB}{\mathbb{B}}
\newcommand{\unit}{\odot}
\newcommand{\R}{\mathbb{R}}
\newcommand{\Z}{\mathbb{Z}}
\newcommand{\N}{\mathbb{N}}
\newcommand{\RZ}{\mathbb{T}}

\newcommand{\im}{\mathop{\mathrm{im}}}
\newcommand{\Tor}{\mathop{\mathrm{Tor}}}
\newcommand{\C}[1]{\mathsf{C}(#1)}
\newcommand{\OC}[1]{\mathsf{OC}(#1)}
\newcommand{\sus}{\mathsf{S}}

\begin{document}

\title{The surjection property and computable type}

\author{Djamel Eddine Amir and Mathieu Hoyrup\\Universit\'e de Lorraine, CNRS, Inria, LORIA, Nancy, 54000, France}

\maketitle
\begin{abstract}
We provide a detailed study of two properties of spaces and pairs of spaces, the surjection property and the~$\epsilon$-surjection property, that were recently introduced to characterize the notion of computable type arising from computability theory. For a class of spaces including the finite simplicial complexes, we develop techniques to prove or disprove these properties using homotopy and homology theories, and give applications of these results. In particular, we answer an open question on the computable type property, showing that it is not preserved by taking products. We also observe that computable type is decidable for finite simplicial complexes.
\end{abstract}

\section{Introduction}

The surjection property and the~$\epsilon$-surjection property are topological properties of spaces and pairs of spaces, introduced in \cite{AH22} to derive topological characterizations of a property from computability theory first studied by Miller \cite{Miller02} and Iljazovi\'c \cite{Iljazovic13}, called \emph{computable type}. Informally, a compact metrizable space~$X$ has computable type if for every homeomorphic copy~$X'$ of the space in the Hilbert cube~$[0,1]^\N$, if the metric balls that are disjoint from~$X'$ can be enumerated by a Turing machine, then the metric balls intersecting~$X'$ can also be enumerated by a Turing machine. Miller \cite{Miller02} proved that~$n$-dimensional spheres have computable type, and Iljazovi\'c \cite{Iljazovic13} proved that closed~$n$-manifolds have computable type for all~$n\in\N$.

It was proved in \cite{AH22} that a finite simplicial complex~$X$ has computable type if and only if it satisfies the~$\epsilon$-surjection property for some~$\epsilon>0$, which means that every continuous function~$f:X\to X$ which is~$\epsilon$-close to the identity must be surjective. It was also proved that this property holds if and only if every star, which is a cone~$\C{L}$, satisfies the surjection property, which means that every continuous function~$f:\C{L}\to\C{L}$ which is the identity on~$L$ is surjective.

The purpose of this article is to further study the surjection and~$\epsilon$-surjection properties, to provide techniques to establish or refute them notably using homotopy and homology theories, and to present applications of these techniques. The results apply to finite simplicial complexes, and slightly larger classes of compact metrizable spaces.

One of the main applications of our results is that the computable type property is not preserved under products, answering a question raised by \v{C}elar and Iljazovi\'{c} in \cite{CI21}. Namely, there exists a finite simplicial complex~$X$ such that~$X$ has computable type, but its product with the circle~$X\times \Sp_1$ does not. This result heavily relies on the properties of the suspension homomorphism between homotopy groups of spheres.

Another application is that the ($\epsilon$-)surjection property holds when the space is a union of homology cycles, implying in particular that in a simplicial complex~$K$, if every~$n$-simplex belongs to an even number of maximal~$(n+1)$-simplices, then~$K$ has computable type.

The reduction to homotopy also implies that the computable type property is decidable for finite simplicial complexes, as a consequence of the decidability of homotopy of simplicial maps \cite{FV20}.

The computable type property, the~$\epsilon$-surjection property and all the results are more generally stated for pairs~$(X,A)$ consisting of a compact metrizable space~$X$ and a compact subset~$A$.

%
%
%
%
%
%
%
%
%
%

Some of the results were announced in the article \cite{AH22} for simplicial complexes and are presented here with complete proofs and for slightly more general spaces.

The article is organized as follows. In Section \ref{sec_surjection_property} we investigate general aspects of the ($\epsilon$-)surjection property, such as their preservation under countable unions. In Section \ref{sec_conical} we show that for spaces that are finite unions of cones, the~$\epsilon$-surjection property is equivalent to the surjection property of each cone. In Section \ref{sec_cones} we investigate the surjection property of cones and obtain precise relationships with homotopy and homology theories.  In Section \ref{sec_counter_examples} we apply these results to a family of spaces which is a source of counter-examples. In Section \ref{sec_computable_type} we give applications to the computabe type property. We end with a conclusion in Section \ref{sec_conclusion}.

We include an appendix which contains results that are classical or folklore, but which are stated slightly differently in the literature, and therefore deserve some explanations. Section \ref{sec_hom} contains results about homotopy of functions, Section \ref{sec_cone_join} is about the cone and join of spaces, Section \ref{sec_hopf} contains Hopf's extension and classification theorems and Section \ref{sec_coef} discusses the choice of the coefficients in homology groups.

\subsection{Preliminaries}

We will work with compact metrizable spaces only. A \textbf{compact pair}~$(X,A)$ consists of a compact metrizable space~$X$ and a compact subset~$A\subseteq X$.

Every compact metrizable space embeds in the \textbf{Hilbert cube}~$Q=[0,1]^\N$ endowed with the product topology, and with the complete metric~$d(x,y)=\sum_{i\in\N}2^{-i}|x_i-y_i|$, where~$x=(x_i)_{i\in\N}$ and~$y=(y_i)_{i\in\N}$. If~$X$ is a compact space and~$f,g:X\to Q$ are continuous functions, then we let~$d_X(f,g)=\max_{x\in X} d(f(x),g(x))$. 

If~$(X,d)$ is a metric space,~$A\subseteq X$ and~$r>0$, let~$\Ne(A,r)=\{x\in X:d(x,A)<r\}$.

If two topological spaces~$X$ and~$Y$ are homeomorphic, we write~$X\cong Y$.

A compact metrizable space~$X$ is an \textbf{Absolute Neighborhood Retract (ANR)} if every (equivalently, some) homeomorphic copy~$X_0\subseteq Q$ of~$X$ is a retract of an open set~$U\subseteq Q$ containing~$X_0$.

\subsubsection{\label{sec_euclidean}Euclidean points and regular cells}
We introduce a class of spaces for which most points have a Euclidean neighborhood. Most of the results will apply to these spaces.
\begin{definition}
Let~$X$ be a topological space and~$n\in\N$,~$n\geq 1$. A point~$x\in X$ is \textbf{$n$-Euclidean} if it has an open neighborhood that is homeomorphic to~$\R^n$. A point is \textbf{Euclidean} if it is~$n$-Euclidean for some~$n$.

A space is \textbf{almost Euclidean} if the set of Euclidean points is dense. A space is \textbf{almost $n$-Euclidean} if the set of~$n$-Euclidean points is dense. A pair~$(X,A)$ is almost Euclidean (resp.~almost $n$-Euclidean) if~$X\setminus A$ is almost Euclidean (resp.~almost $n$-Euclidean).
\end{definition}

Every CW-complex without isolated point is almost Euclidean and every~$n$-manifold is almost~$n$-Euclidean.

\begin{definition}
Let~$X$ be a topological space and~$n\geq 1$. A set~$C\subseteq X$ is a~\textbf{regular $n$-cell} if there is a homeomorphism~$f:\BB_n\to C$ such that~$f(\BB_n\setminus \Sp_{n-1})$ is an open subset of~$X$. The set~$f(\Sp_{n-1})$ is the \textbf{border} of the cell and is denoted by~$\bd{C}$. The set~$f(\BB_n\setminus \Sp_{n-1})$ is the corresponding \textbf{open cell} and is denoted by~$\op{C}$. A \textbf{regular cell} is a regular $n$-cell for some~$n\geq 1$.
\end{definition}

Note that~$\bd{C}$ and~$\op{C}$ always contain the topology boundary and the interior of~$C$ respectively, but do not always coincide with them. For instance, in a simplicial complex, a free face of a simplex is contained in its border but not in its boundary.

A point~$x\in X$ is~$n$-Euclidean if and only if~$x$ belongs to~$\op{C}$ for some~$n$-cell~$C\subseteq X$.

\section{The (\texorpdfstring{$\epsilon$}{epsilon}-)surjection property}\label{sec_surjection_property}

The surjection property and the $\epsilon$-surjection property were defined in~\cite{AH22}. We recall their definitions and develop techniques to prove or disprove these properties.
\begin{definition}
A pair $(X,A)$ has the \textbf{surjection property}, if every continuous
function~$f:X\to X$ such that~$f|_{A}=\id_{A}$ is surjective.

Let~$(X,d)$ be a metric space and~$A\subseteq X$ a closed set. For~$\epsilon>0$, the pair~$(X,A)$ has the \textbf{$\epsilon$-surjection property} if every continuous function~$f:X\to X$ such that~$f|_{A}=\id_{A}$ and~$d(f,\id_{X})<\epsilon$ is surjective. The space~$X$ has the~$\epsilon$-surjection property if the pair~$(X,\emptyset)$ does.
\end{definition}

Of course, the surjection property implies the~$\epsilon$-surjection property for any~$\epsilon>0$.

\begin{example}[Ball]\label{example ball sphere}
For every~$n\in\mathbb{N}$, the pair~$(\BB_{n+1},\Sp_{n})$ has the surjection property. It is a consequence of  the fact that~$\Sp_n$ is not a retract of~$\BB_{n+1}$, but is a retract of any proper subset of~$\BB_{n+1}$ containing~$\Sp_n$.

\end{example}

\begin{example}[Sphere]
Let~$n\in\N$ and let~$d$ be a compatible metric on~$\Sp_n$. If~$\epsilon$ is smaller than half the distance between every pair of antipodal points, then~$\Sp_n$ has the~$\epsilon$-surjection property. It is a consequence of Borsuk-Ulam's theorem.
\end{example}

Although the~$\epsilon$-surjection property depends on the metric, quantifying over~$\epsilon$ only depends on the topology induced by the metric.
\begin{proposition}[Proposition~3.3 in~\cite{AH22}]
Let~$(X,A)$ be a compact pair. Whether there exists~$\epsilon>0$ such that~$(X,A)$ has the~$\epsilon$-surjection property does not depend on the choice of a compatible metric on~$X$. 
\end{proposition}

When the pair is well-behaved, the condition~$f|_A=\id_A$ can be replaced by the weaker condition~$f(A)\subseteq A$.

\begin{lemma}\label{lem_id_vs_contained}
Let~$(X,A)$ be a compact pair satisfying the homotopy extension property and assume that~$A$ is an ANR whose interior is empty. The following statements are equivalent:
\begin{enumerate}
\item There exists~$\epsilon>0$ such that~$(X,A)$ has the~$\epsilon$-surjection property,
\item There exists~$\epsilon>0$ such that every continuous function~$f:(X,A)\to (X,A)$ satisfying~$d(f,\id_X)<\epsilon$ is surjective.
\end{enumerate}
\end{lemma}
When~$A$ has non-empty interior, the result still holds if one replaces the surjectivity of~$f$ by~$X\setminus A\subseteq\im(f)$ in condition 2.

\begin{proof}
Of course~$2.$ implies~$1.$, we prove the other direction. We assume that~$2.$ does not hold and prove that~$1.$ does not hold. Let~$\epsilon>0$. Let~$\delta<\epsilon/2$ be such that functions to~$A$ that are~$\delta$-close are~$\epsilon/2$-homotopic (Lemma \ref{lem_close_homotopic}). Let~$f:(X,A)\to (X,A)$ be a non-surjective function satisfying~$d(f,\id_X)<\delta$. There is an~$\epsilon/2$-homotopy~$h_t:A\to A$ from~$h_0=f|_A$ to~$h_1=\id_A$. We then apply Lemma \ref{lem_controlled_hep}, which implies that there is an~$\epsilon/2$-homotopy~$H_t:X\to A\cup \im(f)$ extending~$h_t$, from~$f$ to some~$g:X\to A\cup \im(f)$. As~$A$ has empty interior,~$X$ is compact and~$f$ is not surjective,~$A\cup\im(f)$ is a proper subset of~$X$ so~$g$ is a non-surjective function to~$X$. One has~$g|_A=h_1=\id_A$ and~$d(g,\id_X)\leq d(g,f)+d(f,\id_X)<\epsilon$, so~$(X,A)$ does not satisfy the~$\epsilon$-surjection property.
%
%
%
\end{proof}

The product of two pairs is~$(X,A)\times (Y,B)=(X\times Y,X\times B\cup A\times Y)$. For well-behaved pairs, the~$\epsilon$-surjection property of the product implies the~$\delta$-surjection property of the two pairs for some~$\delta$. 

\begin{proposition}[Product]\label{prop_product}
Let~$X,Y$ and~$A\subsetneq X,B\subsetneq Y$ be compact ANRs.

If~$(X,A)\times (Y,B)$ has the~$\epsilon$-surjection property, then~$(X,A)$ and~$(Y,B)$ satisfy the~$\delta$-surjection for some~$\delta>0$.
\end{proposition}
\begin{proof}
Let~$(Z,C)=(X,A)\times (Y,B)=(X\times Y,X\times B\cup A\times Y)$. Note that~$C$ is a compact ANR: indeed, the product of ANRs is an ANR (Theorem IV.7.1 in \cite{Borsuk67}), so~$X\times B$,~$A\times Y$ and their intersection~$A\times B$ are ANRs, therefore their union is an ANR (Theorem IV.6.1 in \cite{Borsuk67}). We can assume w.l.o.g.~that the metric on the product space is the maximum of the metrics on~$X$ and~$Y$.

We assume that~$(X,A)$ does not satisfy the~$\epsilon$-surjection property for any~$\epsilon>0$, and show that the same holds for~$(Z,C)$. Let~$\epsilon>0$ and~$f:X\to X$ be a non-surjective function such that~$f|_A=\id_A$ and~$d(f,\id_X)<\epsilon$. We naturally define~$g:(Z,C)\to (Z,C)$ by~$g(x,y)=(f(x),y)$. One easily checks that~$g(C)\subseteq C$. As~$d(f,\id_X)<\epsilon$, one has~$d(g,\id_Z)<\epsilon$. The image of~$g$ does not contain~$Z\setminus C$: if~$x_0\in X\setminus \im(f)$ and~$y_0\in Y\setminus B$, then~$(x_0,y_0)\notin\im(g)$. Therefore, we can apply Lemma \ref{lem_id_vs_contained}, implying that~$(Z,C)$ does not satisfy the~$\epsilon$-surjection property for any~$\epsilon>0$.
%
%
\end{proof}

\subsection{The (\texorpdfstring{$\epsilon$}{epsilon}-)surjection property for cone pairs}

%
%
%

Cones have the particular property that they contain arbitrarily small copies of themselves, implying that for any~$\epsilon>0$, the~$\epsilon$-surjection is equivalent to the surjection property. Cones and cone pairs are defined in Section \ref{sec_cone}.
\begin{proposition}\label{prop_cone_epsilon}
Let~$(X,A)$ be a compact pair and~$\epsilon>0$. The cone pair~$\C{X,A}$ has the~$\epsilon$-surjection property iff it has the surjection property.
\end{proposition}

\begin{proof}
Assume that~$X$ is embedded in~$Q$ and let~$\C{X}=\{(t,t x):t\in[0,1],x\in X\}$. Assume that there is a non-surjective continuous function~$f:\C{X}\to\C{X}$ which is the identity on~$X\cup\C{A}$. For any~$\epsilon>0$, one can define such a function~$g$ which is~$\epsilon$-close to the identity. We decompose~$\C{X}$ as~$\C{X}=C\cup D$ where
\begin{align*}
C&=\{(t,tx):t\in[0,\delta],x\in X\}\\
D&=\{(t,tx):t\in[\delta,1],x\in X\}.
\end{align*}
Note that~$C$ is homeomorphic to~$\C{X}$. We define~$g$ as the identity on~$D$ and as a rescaled version of~$f$ on~$C$, namely~$g(\delta t,\delta tx)=\delta f(t,tx)$ for~$t\in [0,1]$ and~$x\in X$. $g$ is non-surjective, is the identity on~$X\cup\C{A}$, and if~$\delta$ is sufficiently small, then~$g$ is~$\epsilon$-close to the identity.

In Section \ref{sec_cone}, we introduce the symbol~$\unit$ and define~$\C{\unit}=\{0\}$. Note that the proof works when~$A=\unit$ and~$\C{X,\unit}=(\C{X},X\cup \{0\})$ (where~$0$ is the tip of~$\C{X}$).
\end{proof}

The only case when we need to consider~$A=\unit$ is when~$X$ is a singleton.
\begin{proposition}
If~$X$ is not a singleton, then~$\C{X,\unit}$ has the surjection property if and only if~$\C{X,\emptyset}$ has the surjection property.
\end{proposition}
\begin{proof}
Assume that~$\C{X,\emptyset}$ does not have the surjection property, and let~$f:\C{X}\to\C{X}$ be a non-surjective continuous function which is the identity on~$X$. We build a non-surjective continuous function~$g:\C{X}\to\C{X}$ which is the identity on~$X\cup\{0\}$.

We are going to define a proper subspace~$Y\subsetneq \C{X}$ that contains~$\im(f)\cup \{0\}$ and contains a path from~$0$ to~$f(0)$. Let then~$i:X\cup\{0\}\to Y$ be the inclusion and~$j:X\cup \{0\}\to Y$ be the identity on~$X$ and send~$0$ to~$f(0)$. The path from~$0$ to~$f(0)$ in~$Y$ induces a homotopy from~$i$ to~$j$. As the pair~$(\C{X},X\cup\{0\})$ has the homotopy extension property, it implies that the inclusion has a continuous extension~$g:\C{X}\to Y$, which is non-surjective when typed as~$g:\C{X}\to\C{X}$, which completes the proof.

We now define the space~$Y$. There are two cases.

First assume that~$0\in \im (f)$. In that case, we simply take~$Y=\im(f)$. Let~$x\in\C{X}$ be such that~$f(x)=0$. The image by~$f$ of the ray from~$0$ to~$x$ in~$\C{X}$ is a path from~$f(0)$ to~$0$, contained in~$Y$.

Now assume~$0\notin\im(f)$. Let~$Y$ be the union of~$\im(f)$ and the ray from~$0$ to~$f(0)$. We need to show that~$Y$ is a proper subset of~$\C{X}$. Let~$a,b$ be two distinct points of~$X$. As~$\im(f)$ is closed and does not contain~$0$, the rays from~$0$ to~$a$ and~$b$ are not contained in~$\im(f)$. $f(0)$ belongs to at most one of them, so the other ray is still not contained in~$Y$. Therefore,~$Y$ is a proper subset of~$\C{X}$.
\end{proof}

\subsection{The (\texorpdfstring{$\epsilon$}{epsilon}-)surjection property and unions}

In certains cases, the ($\epsilon$-)surjection property for a pair~$(X,A)$ can be established by proving the ($\epsilon$-)surjection property  for subpairs covering~$(X,A)$.

The first result holds for cone pairs and countable unions.

\begin{theorem}\label{thm: infinite union}
Let~$(X,A)$ and~$(X_i,A_i)_{i\in\N}$ be compact pairs such that~$X=\bigcup_{i\in\N} X_i$ and~$A=\bigcup_{i\in\N} A_i$. Assume that~$(X,A)$ is almost Euclidean.

If every pair $\C{X_{i},A_{i}}$ has the surjection property, then~$\C{X,A}$ has the surjection property.
\end{theorem}

\begin{proof}
Assume that~$\C{X,A}$ does not have the surjection property. There exist~$n\in\N$ and a regular~$n$-cell~$C\subseteq X\setminus A$ such that the corresponding quotient map~$q_C:(X,A)\to(\Sp_n,s)$ is null-homotopic.

For each~$i$, the topological boundary~$\partial X_i$ of~$X_i$ is nowhere dense in~$X$, so~$\bigcup_{i\in\N}\partial X_i$ is meager. By the Baire category theorem, its complement is dense, in particular it intersects~$\op{C}$. Let~$x$ belong to the intersection and let~$i$ be such that~$x\in X_i$. As~$x\notin\partial X_i$,~$x$ belongs to the interior of~$X_i$ so~$x$ is~$n$-Euclidean in~$X_i$. Note that~$x\notin A_i$, because~$C$ is disjoint from~$A$.

Let~$C'\subseteq C\cap X_i\setminus A_i$ be a regular~$n$-cell. The quotient map~$q_{C'}:(X,A)\to (\Sp_n,s)$ is homotopic to~$q_{C}$ so it is null-homotopic. Therefore, its restriction to~$(X_i,A_i)$ is null-homotopic, implying that~$\C{X_i,A_i}$ does not have the surjection property.
\end{proof}

Under certain conditions, the~$\epsilon$-surjection property is preserved by taking finite unions. This result was stated as Theorem 4.1 in \cite{AH22} for simplicial complexes.

\begin{theorem}\label{thm_finite_union}
Let~$(X,A)$ and~$(X_i,A_i)_{i\leq n}$ be compact pairs such that~$X=\bigcup_{i\leq n}X_i$ and~$A=\bigcup_{i\leq n}A_i$. Assume that each topological boundary~$\partial X_i$ is a neighborhood retract in~$X$.

If every pair~$(X_i,A_i)$ has the~$\epsilon$-surjection property for some~$\epsilon>0$, then~$(X,A)$ has the~$\delta$-surjection property for some~$\delta>0$.
\end{theorem}
\begin{proof}
For each~$i$, there is a neighborhood~$U_i$ of~$X_i$ and a retraction~$r_i:U_i\to X_i$ such that the only preimage of each~$x\in \int{X_i}$ is~$x$.

Indeed, let~$V_i\subseteq X$ be a neighborhood of~$\partial X_i$ and~$\rho_i:V_i\to\partial X_i$ be a retraction. Let~$U_i=X_i\cup V_i=\int{X_i}\cup V_i$ and define~$r_i:U_i\to X_i$ as the identity on~$X_i$ and as~$\rho_i$ on~$V_i\setminus \int{X_i}$. The definition is consistent because the intersection of~$X_i$ and~$V_i\setminus \int{X_i}$ is~$\partial X_i$, on which~$\rho_i$ coincides with the identity. Therefore,~$r_i$ is well-defined and continuous, and is indeed a retraction.

Let~$\epsilon>0$ be such that each~$(X_i,A_i)$ has the~$\epsilon$-surjection property. Let~$\delta>0$ be such that for each~$i\leq n$,~$\Ne(X_i,\delta)\subseteq U_i$ and for~$x,y\in U_i$,~$d(x,y)<\delta$ implies~$d(r_i(x),r_i(y))<\epsilon$.

Assume the existence of a non-surjective continuous function~$f:X\to X$ such that~$f|_A=\id_A$ and~$d(f,\id_X)<\delta$. There exists~$i\leq n$ and~$x_0\in \int{X_i}$ which is not in the image of~$f$. Let~$f_i=r_i\circ f|_{X_i}:X_i\to X_i$. It is well-defined, because~$f(X_i)\subseteq \Ne(X_i,\delta)\subseteq U_i$. Both~$f$ and~$r_i$ are the identity on~$A$, so~$f_i$ is the identity on~$A$. By choice of~$\delta$,~$d(f_i,\id_{X_i})<\epsilon$. Finally,~$f_i$ is not surjective, because~$x_0$ is not in its image: the only preimage of~$x_0\in\int{X_i}$ by~$r_i$ is~$x_0$, which is not in the image of~$f$.
\end{proof}
\section{Finitely conical spaces}\label{sec_conical}
We show that for a class of spaces containing finite simplicial complexes and compact manifolds, the~$\epsilon$-surjection property is actually a local property.

\subsection{Definition and first properties}
\begin{definition}\label{def_locally_conical}
A topological space~$X$ is \textbf{finitely conical} if there exists a finite sequence of compact metrizable spaces~$(L_i)_{i\leq n}$ and a finite covering of~$X$ by open sets~$U_i\subseteq X$,~$i\leq n$, where~$U_i$ is homeomorphic to~$\OC{L_i}$.

A pair~$(X,A)$ is \textbf{finitely conical} if~$X$ is finitely conical and for every~$i$, the homeomorphism~$f_i:U_i\to\OC{L_i}$ satisfies~$f(U_i\cap A)=\OC{N_i}$ for some compact~$N_i\subseteq L_i$ (or~$N_i=\unit$).
\end{definition}
We will say that the sequence~$(L_i)_{i\leq n}$ or~$(L_i,N_i)_{i\leq n}$ witnesses the fact that~$X$ or~$(X,A)$ is finitely conical.

For simplicity of notation, we will identify~$U_i$ with~$\OC{L_i}$, so the condition for pairs can be written as~$\OC{L_i}\cap A=\OC{N_i}$.

\begin{remark}[Particular case]
We allow the pair~$(L_i,N_i)=(\{0\},\unit)$, giving the open cone pair~$\OC{\{0\},\unit}=([0,1),\{0\})$.
\end{remark}

\begin{example}[Graph]
Let~$X$ be a finite topological graph and~$A$ be the set vertices of degree~$1$. The pair~$(X,A)$ is finitely conical. An open cone consists of a vertex~$v$ together with the open edges starting at~$v$. The open cone pair centered at~$v$ is~$\OC{L,N}$, where~$L$ is the set of vertices that are neighbors of~$v$ and~$N=\unit$ if~$v$ has degree~$1$, and~$N=\emptyset$ otherwise.
\end{example}

\begin{example}[Simplicial pair]
More generally, every pair~$(X,A)$ consisting of a finite simplicial complex~$X$ and a subcomplex~$A$ is finitely conical, witnessed by the open stars of the vertices.
\end{example}

\begin{example}[Compact manifold]
Every compact manifold of dimension~$n\in\mathbb{N}^{\ast}$ with possibly empty boundary~$\partial M$ is finitely conical. The pairs~$(L_i,N_i)$ are~$(\Sp_{n-1},\emptyset)$, as well as~$(\BB_{n-1},\Sp_{n-2})$ when~$\partial M\neq\emptyset$. A point~$x\in M\setminus \partial M$ is the tip of~$\OC{\Sp_{n-1}}$, a point~$x\in \partial M$ is the tip of~$\OC{\BB_{n-1}}$. Note that for~$n=1$, one has~$(\BB_{0},\Sp_{-1})=(\{0\},\unit)$ (see Section \ref{sec_cone}).
\end{example}

The class of finitely conical spaces or pairs is preserved by many constructs.
\begin{proposition}\label{prop_loc_con_product}
Finitely conical pairs are closed under finite products and the cone operator.
\end{proposition}

\begin{proof}
Let~$(X_1,A_1)$ and~$(X_2,A_2)$ be finitely conical, and let~$(L_{i}^1,N_{i}^1)_{i\leq m}$ and~$(L_j^2,N_j^2)_{j\leq n}$ be respective witnesses. It is not difficult to see that the~$(X,A)=(X_1,A_1)\times (X_2,A_2)$ is finitely conical, witnessed by the pairs
\begin{equation*}
(L_{i,j},N_{i,j})=(L^1_i,N^1_i)*(L^2_j,N^2_j),
\end{equation*}
where the join~$*$ of pairs is defined in Section \ref{sec_join}.

Indeed, one has
\begin{align*}
X =X_{1}\times X_{2}=\bigcup_{i}\OC{L_{i}^{1}}\times \bigcup_{j}\OC{L_{j}^{2}}& =\bigcup_{i,j}\OC{L_{i}^{1}}\times\OC{L_{j}^{2}} =\bigcup_{i,j}\OC{L_{i,j}},
\end{align*}
by Proposition~\ref{prop_cone_product}. Note that each~$\OC{L_{i,j}}$ is open in~$X$. Moreover,
\begin{align*}
\OC{L_{i,j}}\cap A & =(\OC{L_{i}^{1}}\times\OC{L_{j}^{2}})\cap((A_{1}\times X_{2})\cup(X_{1}\times A_{2}))\\
 & =((\OC{L_{i}^{1}}\cap A_{1})\times\OC{L_{j}^{2}})\cup(\OC{L_{i}^{1}}\times(\OC{L_{j}^{2}}\cap A_{2}))\\
 & =(\OC{N_{i}^{1}}\times\OC{L_{j}^{2}})\cup(\OC{L_{i}^{1}}\times\OC{N_{j}^{2}})\\
 &=\OC{N_i^1*L^2_j\cup L^1_i*N^2_j}\\
 & =\OC{N_{i,j}},
\end{align*}
where the last equality holds by Proposition \ref{prop_cone_pair_product}. We have proved that~$(X,A)$ is finitely conical.

We now show that if~$(X,A)$ is finitely conical, coming with~$(L_{i},N_{i})_{i\leq n}$, then~$\C{X,A}$ is finitely conical, witnessed by the pairs~$\C{L_i,N_i}$ together which the pair~$(X,A)$.

The first component of the pair~$\C{X,A}$ is the cone~$\C{X}$, which is covered by two open sets~$\OC{X}$ and~$X\times (0,1]$. Its second component is~$X\cup \C{A}$, which is covered by the two open sets~$\OC{A}$ and~$X\cup (A\times (0,1])$. All in all, we have~$\C{X,A}=\OC{X,A}\cup ((X,A)\times ((0,1],\{1\}))$.

The first pair~$\OC{X,A}$ is finitely conical, because it consists of one open cone pair.

The second pair~$(X,A)\times ((0,1],\{1\})=(X,A)\times \OC{\{0\},\unit}$ is a product of two finitely conical pairs, so it is finitely conical by the first statement. It is witnessed by the pairs~$(L_i,N_i)*(\{0\},\unit)=\C{L_i,N_i}$. Therefore,~$\C{X,A}$ is the union of two open subspaces which are both finitely conical, so it is finitely conical.
\end{proof}


\begin{remark}
By similar arguments, it can be proved that if the pairs~$(X_1,A_1)$ and~$(X_2,A_2)$ are finitely conical, witnessed by~$(L^1_i,N^1_i)_{i\leq m}$ and~$(L^2_j,N^2_j)_{j\leq n}$ respectively, then their join~$(X,A)*(Y,B)$ is finitely conical, witnessed by the pairs
\begin{equation*}
\C{(L^1_i,N^1_i)*(L^2_j,N^2_j)},
\end{equation*}
which means that~$(X,A)*(Y,B)$ is covered by the open cones of these pairs. We will not use this result.
%
\end{remark}
\subsection{The \texorpdfstring{$\epsilon$}{epsilon}-surjection property for finitely conical pairs}

We show that the~$\epsilon$-surjection property of a finitely conical pair reduces to the surjection property of each local cone pair, assuming that the~$L_i$'s are ANRs. This result is particularly useful because it enables one to check a global property by inspecting the local cones independently of each other. The result was stated in \cite{AH22} for the family of finite simplicial complexes.

\begin{theorem}\label{thm_local_surjection}
Let~$(X,A)$ be a finitely conical compact pair coming with~$(L_{i},N_{i})_{i\leq n}$, where each~$L_i$ is an ANR. The following statements are equivalent:
\begin{enumerate}
\item $(X,A)$ has the~$\epsilon$-surjection property for some~$\epsilon>0$, 
\item All the cone pairs~$\C{L_{i},N_{i}}$ have the surjection property.
\end{enumerate}
\end{theorem}

The implication~$2.\Rightarrow 1.$ does not follow from Theorem \ref{thm_finite_union}, because in  a finitely conical pair~$(X,A)$, one has~$A=\bigcup_i\C{N_i}$ while  applying Theorem \ref{thm_finite_union} would require~$A=\bigcup_iL_i\cup\C{N_i}$.

\begin{proof}
$1.\Rightarrow 2.$ Assume that some~$\C{L_i,N_i}$ does not have the surjection property and let~$\epsilon>0$.~$\C{L_i,N_i}$ does not have the~$\epsilon$-surjection property by Proposition \ref{prop_cone_epsilon}, which is witnessed by a non-surjective continuous function~$f:\C{L_i}\to\C{L_i}$. We assume that~$\C{L_i}$ is embedded in~$X$ in such a way that its topological boundary in~$X$ is~$L_i$. We then extend~$f$ as the identity outside~$\C{L_i}$, showing that~$(X,A)$ does not have the~$\epsilon$-surjection property. Note that the extension of~$f$ is indeed continuous because~$f$ is the identity on the boundary of~$\C{L_i}$.

$2.\Rightarrow1.$ We assume that for every~$\epsilon>0$,~$(X,A)$ does not have the~$\epsilon$-surjection property, and we prove that some cone pair~$\C{L_i,N_i}$ does not have the surjection property.

As~$X$ is compact and is covered by finitely many open cones~$\OC{L_i}$,~$X$ is covered by slightly smaller closed cones~$\C{L_i}$ contained in these open cones. Let~$(K_i,M_i)=\C{L_{i},N_{i}}$ be these closed cones, contained in~$X$. $L_{i}$ can be seen as a subset of the Hilbert cube~$Q$, and there is a homeomorphism
\[
g_i:\{(t,t x):t\in[0,1],x\in L_{i}\}\to K_i.
\]
sending~$\{(1,x):x\in L_i\}\cup\{(t,tx):t\in [0,1],x\in N_i\}$ to~$M_i$ (if~$N_i=\unit$, then the latter set is~$\{(1,x):x\in L_i\}\cup\{(0,0)\}$).

For any~$t\in(0,1)$, define the smaller copy~$(K_i(t),M_i(t))$ of~$(K_i,M_i)$ as follows:
\begin{align*}
K_i(t) & =g_i(\{(s,sx):s\in[0,t],x\in L_{i}\}),\\
M_i(t) & =(K_i(t)\cap A)\cup g_i(\{(t,t x):x\in L_{i}\}).
\end{align*}

Let~$t_0<1$ be such that the sets~$K_i(t_0)$ cover~$X$, which exists  as~$X$ is compact. Pick two other real numbers~$t_0<t_1<t_2<1$. Let~$\epsilon>0$ be such that for all~$i$,
\begin{align}
\Ne(K_i(t_{0}),\epsilon) & \subseteq K_i(t_{1}),\label{eq1K}\\
\Ne(K_i(t_{1}),\epsilon) & \subseteq K_i(t_{2}).\label{eq2K}
\end{align}

Let~$\delta<\epsilon$ be smaller than the values provided by Lemma
\ref{lem_ANR_extension} applied to all the compact ANRs~$K_i(t_{2})$
and~$\epsilon$. By assumption,~$(X,A)$ does not have the~$\delta$-surjection
property, i.e.~there exists a non-surjective continuous function~$h:X\to X$
such that~$h|_{A}=\id_{A}$ and~$d_{X}(h,\id_{X})<\delta<\epsilon$.
As~$X=\bigcup_{i}K_i({t_{0}})$ and~$h$ is not surjective, there exists~$i$ such that
\begin{equation}
K_i({t_{0}})\nsubseteq h(X).\label{eq3K}
\end{equation}
Let~$(K,M)=(K_i({t_{2}}),M_i({t_{2}}))\cong \C{L_i,N_i}$. We define a non-surjective continuous function~$G:K\to K$ such that~$G|_{M}=\id_{M}$, showing that~$\C{L_i,N_i}$ does not have the surjection property.

First observe that~$h(K_i({t_{1}}))$ is contained in~$K=K_i(t_2)$, because~$h$
is~$\epsilon$-close to the identity, so~$h(K_i({t_{1}}))\subseteq\Ne(K_i({t_{1}}),\epsilon)\subseteq K$ by \eqref{eq2K}.

We define~$g:K_i({t_{1}})\cup M\to K$ by 
\begin{equation*}
g|_{K_i({t_{1}})}=h|_{K_i({t_{1}})}\text{ and }g|_{M}=\id_{M}.
\end{equation*}
The function~$g$ is well-defined and continuous because~$h$ coincides with the identity on the intersection~$K_i({t_{1}})\cap M\subseteq A$.

We now define a continuous extension~$G:K\to K$ of~$g$. Note that~$g$
is~$\delta$-close to the inclusion~$f:K_i({t_{1}})\cup M\to K$
and~$f$ has a continuous extension~$F:=\id_{K}:K\to K$, so using
Lemma \ref{lem_ANR_extension},~$g$ has a continuous extension~$G:K\to K$
satisfying~$d_{K}(G,\id_{K})<\epsilon$. As~$G$ extends~$g$,~$G|_{M}=\id_{M}$.
It remains to show that~$G$ is not surjective. Indeed,~$K_i({t_{0}})$ is not
contained in~$G(K)$: 
\begin{itemize}
\item $G(K_i({t_{1}}))=h(K_i({t_{1}}))$ which does not contain~$K_i({t_{0}})$
by \eqref{eq3K}, 
\item $G(K\setminus K_i({t_{1}}))\subseteq\Ne(K\setminus K_i({t_{1}}),\epsilon)$
which is disjoint from~$K_i({t_{0}})$ by \eqref{eq1K}.\qedhere
\end{itemize}
\end{proof}

\section{The surjection property for cone pairs}\label{sec_cones}
Theorem \ref{thm_local_surjection} reduces the~$\epsilon$-surjection property to the surjection property for cone pairs. The purpose of this section is to develop techniques to establish whether a cone pair has the surjection property. We will only work with cones of almost Euclidean spaces.

\subsection{Quotients to spheres}
The first characterization of the surjection property for cone pairs involves quotient maps to spheres associated to regular cells.

Let~$(X,A)$ be a compact pair and~$C\subseteq X\setminus A$ a regular~$n$-cell. The quotient of~$X$ by~$X\setminus \op{C}$ is homeomorphic to~$\Sp_n$. We let
\begin{equation}
q_C:(X,A)\to (\Sp_n,s)
\end{equation}
be the quotient map, where~$s\in\Sp_n$ is the image of~$X\setminus \op{C}$ by~$q_C$.
\begin{theorem}[Surjection property vs quotient maps]\label{thm_surj_prop_cone}
Let~$(X,A)$ be an almost Euclidean compact pair. The following statements are equivalent:
\begin{enumerate}
\item $\C{X,A}$ has the surjection property,
\item For every~$n\geq 1$ and every regular~$n$-cell~$C\subseteq X\setminus A$, the quotient map~$q_C:(X,A)\to (\Sp_n,s)$ is not null-homotopic.
\end{enumerate}
\end{theorem}

Actually, the implication~$1.\Rightarrow 2.$ holds without assuming that~$(X,A)$ is almost Euclidean. In order to prove the theorem, we first give a reformulation of the surjection property for cone pairs.
\begin{lemma}\label{lem_surj_prop_cone}
Let~$(X,A)$ be an almost Euclidean compact pair. The following statements are equivalent:
\begin{enumerate}
\item $\C{X,A}$ does not have the surjection property,
\item There exists a regular cell~$C\subseteq X\setminus A$ and a continuous function~$f:\C X\to X\cup \C D$, where~$D=X\setminus \op{C}$, which is the identity on~$X\cup\C{A}$.
\end{enumerate}
\end{lemma}

\begin{proof}
$1.\Rightarrow 2.$ Let~$f:\C{X}\to\C{X}$ be a non-surjective continuous function which is the identity on~$X\cup\C{A}$. The union of regular cells~$C\subseteq X\setminus A$ is dense in~$X\setminus A$, so the union of their cones is dense in~$\C{X\setminus A}$. Therefore, there exists a regular cell~$C\subseteq X\setminus A$ such that~$\C{C}$ is not contained in the image of~$f$. $\C{C}$ is a regular cell in~$\C{X}$, its intersection with~$X\cup \C{A}$ is contained in its border. Let~$x$ belong to the interior of~$\C{C}$ but not to the image of~$f$. There exists a retraction~$r_0:\C{C}\setminus \{x\}\to\bd{\C{C}}$, which extends to a retraction~$r:\C{X}\setminus \{x\}\to X\cup\C{D}$. The composition~$r\circ f:\C{X}\to X\cup \C{D}$ is well-defined, continuous and is the identity on~$X\cup \C{A}$.

$2.\Rightarrow 1.$ Such a function~$f$ is a non-surjective continuous satisfying~$f|_{\C A\cup X}=\id{}_{\C A\cup X}$.
\end{proof}

The existence of a continuous function~$f:\C{X}\to X\cup\C{D}$ which is the identity on~$X$ is nothing else than a null-homotopy of the inclusion map~$i:X\to X\cup\C{D}$. If the pair~$(X,D)$ has the homotopy extension property, then~$X\cup\C{D}$ is homotopy equivalent to~$X/D$ (Proposition 0.17 in \cite{Hatcher02}), which implies that the inclusion map~$i$ is null-homotopic if and only if the quotient map~$q:X\to X/D$ is null-homotopic.

We prove a similar equivalence applicable to pairs~$(X,A)$. Essentially, we need to be careful about how the homotopy equivalence between~$X\cup\C{D}$ and~$X/D$ is defined on~$A$.

\begin{proposition}\label{prop_quotient_map}
Let~$(X,D)$ be a compact pair satisfying the homotopy extension property and~$A$ a compact subset of~$D$. The following statements are equivalent
\begin{enumerate}
\item There exists a continuous function~$f:\C X\to X\cup\C D$ which is
the identity on~$X\cup\C A$,
\item The quotient map~$q:(X,A)\to(X/D,p)$ is null-homotopic ($p$ being the equivalence class of~$D$).
\end{enumerate}
\end{proposition}

\begin{proof}
$1.\Rightarrow 2.$ The function~$f$ is a null-homotopy for pairs of the inclusion map~$i_{X}:(X,A)\to (X\cup\C D,\C{A})$. Consider the quotient map~$p:X\cup\C D\to(X\cup\C D)/\C D$ and observe that~$(X\cup\C D)/\C D$ can be identified with~$X/D$. The function~$p\circ i_{X}:(X,A)\to (X/D,p)$ is precisely~$q$, and the null-homotopy of~$i_{X}$ composed with~$p$ is a null-homotopy of~$q$.

$2.\Rightarrow 1.$ Now assume that~$q$ is null-homotopic.

First, we show that the inclusion~$i_{X}:X\to X\cup\C D$
is null-homotopic. The obvious null-homotopy of the inclusion~$i_D:D\to X\cup \C{D}$ can be extended to a homotopy
\begin{equation*}
F_t:X\to X\cup \C{D}
\end{equation*}
from~$F_0=i_X$ to some~$F_1$ which is constant on~$D$. Let
\begin{equation*}
\widetilde{F_{1}}:X/D\to X\cup\C D
\end{equation*}
be the continuous map induced by~$F_{1}$, i.e.~satisfying~$F_{1}=\widetilde{F_{1}}\circ q$. Let
\begin{equation*}
H_{t}:(X,A)\to (X/D,p)
\end{equation*}
be a homotopy between~$H_{0}=q$ and the constant map~$H_{1}=p$. Observe that~$\widetilde{F_{1}}\circ H_{t}:X\to X\cup\C D$ is a homotopy from~$\widetilde{F_{1}}\circ q=F_{1}$ to~$\widetilde{F_{1}}\circ H_{1}$ which is constant. Therefore, the inclusion~$i_{X}:X\to X\cup\C D$ is homotopic to~$F_{1}$ which is null-homotopic, so~$i_{X}$ is
null-homotopic. The null-homotopy of~$i_X$ is given by~$K_{t}=F_{2t}$ for~$0\leq t\leq1/2$
and~$K_{t}=\widetilde{F_{1}}\circ H_{2t-1}$ for~$1/2\leq t\leq 1$.

The null-homotopy~$K_{t}$ can be seen as a function~$g:\C X\to X\cup\C D$ which is the identity on~$X$. However,~$g$ is not the identity on~$\C{A}$, but we show how to modify~$g$.

We use the following notation:~$\C{X}$ is the quotient of~$[0,1]\times X$ obtained by identifying all the points~$(1,x)$, so we can express any point of~$\C{X}$ as the equivalence class of a pair~$(t,x)\in [0,1]\times X$, denoted~$[t,x]$. Note that~$[1,x]=[1,x']$ for all~$x,x'\in X$.

With this notation, we can see how~$g$ is defined on~$\C{A}$. Let~$\tau$ be the tip of~$\C{D}$. For~$x\in A$ and~$t\in[0,1]$, one has
\begin{equation*}
g([t,x])=\begin{cases}
K_{t}(x)=F_{2t}(x)=[2t,x]&\text{if }0\leq t\leq 1/2,\\
K_{t}(x)=\tau&\text{if }1/2\leq t\leq 1.
\end{cases}
\end{equation*}

Observe that~$g$ is constant on the segment~$S=\{[t,x]:t\in [1/2,1],x\in A\}$, and the idea is to contract~$S$ to a point as follows.


Let~$D>\max_{x\in X}d(x,A)$ and~$h:\C X\to\C X$ be the continuous function defined by
\begin{equation*}
h([t,x])=\begin{cases}
[(2-d(x,A)/D)t,x] & \text{if }0\leq t\leq1/2,\\
[1+(d(x,A)/D)(t-1),x] & \text{if }1/2\leq t\leq 1,
\end{cases}
\end{equation*}
and represented on Figure \ref{fig_h}.

\begin{figure}[h]
\centering
\includegraphics{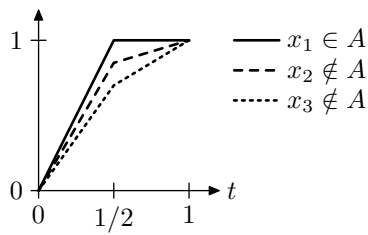}
\caption{Graph of~$t\mapsto h_1(t,x)$, where~$h(t,x)=[h_1(t,x),h_2(t,x)]$, for different~$x$'s}\label{fig_h}
\end{figure}

First,~$h$ is surjective because for each~$x\in X$,~$h$ sends the segment~$\{[t,x]:t\in[0,1]\}$ to itself and fixes its endpoints, so~$h$ sends the segment onto itself. As~$\C{X}$ is compact Hausdorff,~$h$ is a quotient map.

We claim that the function~$f:\C{X}\to X\cup\C{D}$ satisfying~$g=f\circ h$ is well-defined. Indeed, if~$h(u)=h(v)$, then~$u=v$ or~$u,v\in S$, therefore~$g(u)=g(v)$. As~$h$ is a quotient map and~$g$ is continuous,~$f$ is continuous as well.

The restriction of~$f$ to~$X$ (identified with~$\{[0,x]:x\in X\}$) is the identity, because~$h|_X=g|_X=\id_X$. The function~$h$ sends~$\C{A}$ onto~$\C{A}$, and coincides with~$g$ on~$\C{A}$, so the restriction of~$f$ to~$\C{A}$ is the identity.
\end{proof}

\begin{proof}[Proof of Theorem \ref{thm_surj_prop_cone}]
We combine Lemma \ref{lem_surj_prop_cone} and Proposition \ref{prop_quotient_map}, which is possible because the pair~$(X,X\setminus \op{C})$ has the homotopy extension property.
\end{proof}

An interesting consequence is that it is always possible to replace a pair~$(X,A)$ by the single space~$X/A$.
\begin{corollary}[Pair vs quotient]
Let~$(X,A)$ be an almost Euclidean compact pair. The pair~$\C{X,A}$ has the surjection property if and only if~$\C{X/A,\emptyset}=(\C{X/A},X/A)$ has the surjection property.
\end{corollary}
\begin{proof}
Note that~$X/A$ is almost Euclidean, so we can apply Theorem \ref{thm_surj_prop_cone} to both~$(X,A)$ and~$(Y,B)=(X/A,\emptyset)$. For each regular cell~$C\subseteq X\setminus A$, the corresponding quotient map~$q:(X,A)\to (\Sp_n,s)$ is null-homotopic if and only if the quotient map~$q':X/A\to\Sp_n$ is null-homotopic, applying Proposition \ref{prop_quotient_pair}. Therefore, Theorem \ref{thm_surj_prop_cone} applied to~$(X,A)$ and~$(Y,B)$ gives the equivalence.
\end{proof}

Another consequence is that in Proposition \ref{prop_quotient_map}, the condition that~$f:\C{X}\to\C{X}$ is the identity on~$X\cup \C{A}$ can be equivalently replaced by an apparently weaker condition.
\begin{corollary}
Let~$(X,D)$ satisfy the homotopy extension property and~$A$ a compact subset of~$D$. The following statements are equivalent:
\begin{itemize}
\item There exists a continuous function~$f:\C{X}\to X\cup\C{D}$ which is the identity on~$X\cup\C{A}$,
\item There exists a continuous function~$f:\C{X}\to X\cup\C{D}$ which is the identity on~$X$ and satisfies~$f(\C{A})\subseteq\C{D}$.
\end{itemize}
\end{corollary}
\begin{proof}
The second condition precisely means that the inclusion map
\begin{equation*}
i:(X,A)\to (X\cup\C{D},\C{D})
\end{equation*}
is null-homotopic. It implies that the quotient map~$q:(X,A)\to (X/D,p)$ is null-homotopic, which in turn implies the first condition by Proposition \ref{prop_quotient_map}.
\end{proof}

\subsection{Retractions to spheres}
We give a simple sufficient condition implying that~$\C{X,A}$ does not have the surjection property. We will later show that this condition is also necessary under certain assumptions.
\begin{theorem}\label{thm_retraction_quotient}
Let~$(X,A)$ be a compact pair and~$C\subseteq X\setminus A$ be a regular~$n$-cell. If there exists a retraction~$r:X\setminus \op{C}\to \bd{C}$ which is constant on~$A$, then the quotient map~$q:(X,A)\to (\Sp_n,s)$ is null-homotopic, hence~$\C{X,A}$ does not have the surjection property.
\end{theorem}
Note that a retraction~$r:X\setminus \op{C}\to\bd{C}$ is the same thing as a retraction~$r':X\to C$ which sends~$X\setminus \op{C}$ to~$\bd{C}$ (indeed,~$r'$ can be obtained from~$r$ by extending it as the identity on~$C$, and~$r$ can be obtained from~$r'$ by restriction to~$X\setminus \op{C}$).
\begin{proof}
Let~$D=X\setminus \op{C}$ and~$r_0:D\to\bd C$ be a retraction with constant value~$p$ on~$A$. Let~$r:X\to C$ be the retraction extending~$r_0$, defined as the identity on~$C$. As a function of pairs,~$r:(X,A)\to (C,p)$ is null-homotopic. Indeed, there is a contraction of~$C\cong \BB_n$ that fixes~$p$, i.e.~the identity~$\id_C:(C,p)\to (C,p)$ is null-homotopic, so~$r=\id_C\circ r$ is null-homotopic.

The quotient of~$(X,A)$ by~$D$ is~$(\Sp_n,s)$, and the quotient of~$(C,p)$ by~$\bd{C}$ is~$(\Sp_n,s)$. Let~$q_C:(X,A)\to (\Sp_n,s)$ and~$q_C':(C,p)\to (\Sp_n,s)$ be the quotient maps. As~$r$ is the identity on~$C$, one has~$q_C=q_C'\circ r$. Therefore, the null-homotopy of~$r$ induces a null-homotopy of~$q_C$. Theorem \ref{thm_surj_prop_cone} implies that~$\C{X,A}$ does not have the surjection property.
\end{proof}

\begin{example}[Graph]\label{ex_graph}
\label{exa: retraction}Let~$X$ be a topological graph and~$C$ be an edge. If~$C$ is neither contained in a cycle nor in a path from a point of~$A$ to another point of~$A$ (see an example on Figure \ref{fig_graph}), then~$\C{X,A}$ does not have the surjection property because there is a retraction of~$X\setminus \op{C}$ to~$\bd{C}$, as we now explain.  We will see with Theorem \ref{thm_graph} that these conditions are tight.

As~$C$ does not belong to a cycle, its two endpoints~$a,b$ belong to two distinct connected components of~$X\setminus \op{C}$. As there is no path from~$A$ to~$A$ through~$C$, the connected component of~$a$ or~$b$, say~$a$, is disjoint from~$A$. Let~$r:X\setminus \op{C}\to\bd{C}$ send that component to~$a$ and all the rest to the other endpoint~$b$. It is a retraction which is constant on~$A$.

\begin{figure}[h]
\centering
\includegraphics{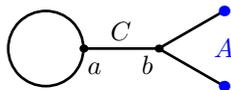}
\caption{A pair~$(X,A)$ whose cone pair~$\C{X,A}$ does not satisfy the surjection property}\label{fig_graph}
\end{figure}
\end{example}

\begin{example}
\label{ex_torus_disk}
Let~$X=\Sp_1\times\Sp_1$ be the torus with a disk attached along one of the two circles, and~$A=\emptyset$. The disk is a regular cell, and the torus retracts to the boundary of the disk, so~$\C{X,A}=(\C{X},X)$ does not have the surjection property.

\begin{figure}[h]
\centering
\includegraphics[height=2cm]{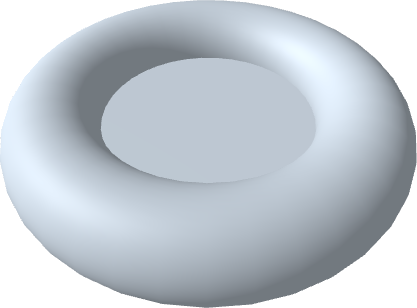}
\caption{A space~$X$ whose cone pair~$(\C{X},X)$ does not satisfy the surjection property}
\end{figure}
\end{example}

The next result shows that the assumptions in Theorem \ref{thm_retraction_quotient} can be weakened.
\begin{proposition}
Let~$(X,A)$ be a compact pair and~$C\subseteq X$ be a regular cell which is not contained in~$A$. If there exists a retraction~$r:X\setminus \op{C}\to\bd C$ whose restriction~$r|_A:A\to\bd C$ is null-homotopic, then~$\C{X,A}$ does not have the surjection property.
\end{proposition}
\begin{proof}
We show that there is a regular cell~$C'\subseteq C\setminus A$ and a retraction~$r':X\setminus \op{C'}\to\bd{C'}$ which is constant on~$A$, and then apply Theorem \ref{thm_retraction_quotient}.

As~$C\nsubseteq A$ and~$A$ is closed, there exists a regular cell~$C'\subseteq \op{C}\setminus A$. The retraction~$r:X\setminus \op{C}\to\bd{C}$ can be extended as the identity on~$C\setminus \op{C'}$, giving a retraction~$r_1:X\setminus \op{C'}\to C\setminus \op{C'}$. There is a retraction~$r_2:C\setminus \op{C'}\to\bd{C'}$, because~$C\setminus \op{C'}$ is a thickened sphere and~$\bd{C'}$ is its inner sphere.

Let~$f=r_2\circ r_1:X\setminus \op{C'}\to\bd{C'}$. The restriction of~$f$ to~$A\cup\bd{C'}$ is homotopic to a function~$h:A\cup\bd{C'}\to\bd{C'}$ which is constant on~$A$ and is the identity on~$\bd{C'}$ (extend the null-homotopy of~$r|_A$ as the identity on~$\bd{C'}$, which is disjoint from~$A$). As~$\bd{C'}$ is an ANR (it is indeed a sphere), Borsuk's homotopy extension theorem implies that~$f$ is homotopic to an extension of~$h$. Such an extension is a retraction~$r':X\setminus \op{C'}\to\bd{C'}$ which is constant on~$A$.
\end{proof}

\subsection{Cycles}
It was proved in \cite{AH22} that when the space~$X$ is a finite topological graph, the surjection property of its cone can be characterized in terms of the cycles of~$X$.
\begin{theorem}[Theorem 4.4 in \cite{AH22}]\label{thm_graph}
Let~$(X,A)$ be a pair consisting of a topological finite graph~$X$ and a subset~$A$ of vertices.~$\C{X,A}$ has the surjection property if and only if every edge belongs to a cycle or a path between two points of~$A$.
\end{theorem}
It is assumed that cycles and paths visit each edge at most once. The result should be compared with Example \ref{ex_graph}.

We investigate how far this result extends to higher-dimensional spaces like finite simplicial complexes, and more generally almost Euclidean spaces. The higher-dimensional notions of cycles and paths between points of~$A$ are the homology cycles and cycles relative to~$A$, i.e.~the elements of~$H_n(X,A)$. We need to express the idea that a cell ``belongs'' to a relative cycle.

Let~$\RZ=\R/\Z$ be the circle group.
\begin{definition}\label{def_cycle}
Let~$(X,A)$ be a compact pair and~$n\geq 1$. A regular~$n$-cell~$C\subseteq X\setminus A$ \textbf{belongs to a relative cycle} if the canonical homomorphism
\begin{equation*}
H_n(X,A;\RZ)\to H_n(X,X\setminus \op{C};\RZ)
\end{equation*}
is non-trivial.
\end{definition}

Let us explain why this definition is meaningful. First, we use the coefficient group~$\RZ$ because it subsumes the more usual coefficients groups~$\Z$ and~$\Z/k\Z$ (see Proposition \ref{prop_coefficients}). The canonical homomorphism is non-trivial if there exists a relative cycle in~$(X,A)$ which is not cancelled when quotienting by the part of the space which is not in~$C$, so the cycle should ``visit''~$C$. An alternative way of seeing this is to consider the exact sequence (with coefficients in~$\RZ$)
\begin{equation*}
H_n(X\setminus \op{C},A)\to H_n(X,A)\to H_n(X,X\setminus \op{C}).
\end{equation*}
The second homomorphism is non-trivial if and only if the first one is non-surjective, which means that there is a relative cycle in~$(X,A)$ which is not contained in~$X\setminus \op{C}$.

\begin{remark}[Homology vs cohomology]
Here is an equivalent formulation of Definition \ref{def_cycle}, at least for pairs~$(X,A)$ where~$X$ and~$A$ are ANRs. $C$ belongs to a relative cycle if and only if the canonical homomorphism on cohomology groups with coefficients in~$\Z$
\begin{equation*}
H^n(X,X\setminus \op{C})\to H^n(X,A)
\end{equation*}
is non-trivial. Indeed,~$H_n(X,A;\RZ)$ is the Pontryagin dual of~$H^n(X,A)$ and the homomorphisms between homology and cohomology groups are dual to each other (Proposition G, \S VIII.4, p.~137 and Proposition F, \S VIII.5, p. 141 in \cite{HurWal41}).
\end{remark}

\subsubsection{Simplicial complexes}
Let us illustrate Definition \ref{def_cycle} in the case of finite simplicial complexes. A finite simplicial pair is a pair~$(X,A)$ where~$X$ is a finite simplicial complex and~$A$ is a subcomplex.

A finite simplicial complex is almost Euclidean, each maximal simplex being a regular cell. For simplicial complexes, singular homology is isomorphic to simplicial homology. Whether a simplex belongs to a relative cycle in the sense of Definition \ref{def_cycle} can be reformulated in a much more explicit way.


\begin{proposition}\label{prop_cycle_simplicial}
Let~$(X,A)$ be a finite simplicial pair and~$C$ a maximal simplex, of dimension~$n$. $C$ belongs to a relative cycle if and only if there exists an~$n$-chain with coefficients in~$\RZ$ (equivalent, in~$\Z$ or some~$\Z/k\Z$) whose boundary is contained in~$A$, assigning a non-zero coefficient to~$C$.
\end{proposition}
\begin{proof}

Let~$(\Delta_i)_{i\leq m}$ be the~$n$-simplices in~$X$, and assume w.l.o.g.~that~$\Delta_0$ is~$C$. Let~$D$ be the subcomplex obtained by removing~$C$ from~$X$ (but keeping its faces). The homomorphism
\begin{equation*}
H_n(X,A;\RZ)\to H_n(X,D;\RZ)
\end{equation*}
 is realized, at the level of~$n$-chains, by sending~$\sum \alpha_i\Delta_i$ to~$\alpha_0\Delta_0$. As~$\Delta_0$ is maximal,~$\Delta_0$ is not a boundary. Therefore, the homomorphism is non-trivial if and only if there exists a relative cycle assigning a non-zero coefficient~$\alpha_0$ to~$\Delta_0$.
\end{proof}

\subsubsection{Cycles vs the surjection property}
Let~$(X,A)$ be a compact pair and~$C\subseteq X\setminus A$ be a regular~$n$-cell ($n\geq 1$). One has natural isomorphisms
\begin{align*}
H_n(X,X\setminus \op{C};\RZ)&\cong H_n(C,\bd{C};\RZ)&\text{(by excision)}\\
&\cong H_n(C/\bd{C},s;\RZ)\\
&\cong H_n(\Sp_n,s;\RZ)\\
&\cong \RZ.
\end{align*}

Therefore, up to isomorphism, the canonical homomorphism from Definition \ref{def_cycle}
\begin{equation*}
H_n(X,A;\RZ)\to H_n(X,X\setminus \op{C};\RZ)
\end{equation*}
is nothing else than the homomorphism induced by the quotient map~$q_C:(X,A)\to (\Sp_n,s)$ on homology groups. It has the following immediate consequences:
\begin{itemize}
\item If~$C$ belongs to a relative cycle, then~$q_C:(X,A)\to (\Sp_n,s)$ is not null-homotopic,
\item When~$\dim(X)=n$ and~$X,A$ are ANRs,~$C$ belongs to a relative cycle if and only if~$q_C:(X,A)\to (\Sp_n,s)$ is not null-homotopic, by Hopf's classification theorem (Theorem \ref{thm_hopf_classification}).
\end{itemize}

Therefore, we obtain partial generalizations of Theorem \ref{thm_graph} to higher-dimensional spaces. The first one is only an implication.

\begin{corollary}[Cycles vs surjection property, implication]\label{cor_implication}
Let~$(X,A)$ be an almost Euclidean compact pair. If every regular cell~$C\subseteq X\setminus A$ belongs to a relative cycle, then~$\C{X,A}$ has the surjection property.
\end{corollary}


This result becomes an equivalence when the space ``has the same dimension everywhere'', giving a first extension of Theorem \ref{thm_graph}, from graphs to \emph{pure} simplicial complexes, i.e.~simplicial complexes~$X$ whose maximal simplices all have the same dimension.

\begin{corollary}[Cycles vs surjection property, equivalence]\label{cor_equivalence}
Let~$n\geq 1$. Let~$(X,A)$ be an almost~$n$-Euclidean compact pair, where~$X$ and~$A$ are ANRs and~$\dim(X)=n$. The following statements are equivalent:
\begin{itemize}
\item The pair~$\C{X,A}$ has the surjection property,
\item Every regular cell~$C\subseteq X\setminus A$ belongs to a relative cycle.
\end{itemize}
\end{corollary}

Theorem \ref{thm_retraction_quotient} also becomes an equivalence when the dimension of the space matches the dimension of the cell, or the cell has dimension at most~$2$. 
\begin{corollary}[Cycles vs retraction]\label{cor_retraction}
Let~$(X,A)$ be a compact pair, where~$X$ and~$A$ are ANRs and~$\dim(X)=n\geq 1$. For a regular~$n$-cell~$C\subseteq X\setminus A$, the following statements are equivalent:
\begin{itemize}
\item $C$ does not belong to a relative cycle,
\item There exists a retraction~$r:X\setminus \op{C}\to\bd{C}$ which is constant on~$A$.
\end{itemize}

The equivalence holds if~$n=1$ or~$2$, without any dimension assumption about~$X$.
\end{corollary}

\begin{proof}
We assume that~$C$ does not belong to a relative cycle and build a retraction (the other implication is Theorem \ref{thm_retraction_quotient}). Let~$Y=X/A$. The existence of a retraction which is constant on~$A$ is equivalent to the existence of a retraction~$r:Y\to\bd{C}$. Let~$f:\bd{C}\to\bd{C}$ be the identity and~$i:\bd{C}\to Y\setminus \op{C}$ be the inclusion map. Note that~$\bd{C}\cong \Sp_{n-1}$ and let~$f_*,i_*$ be the induced homomorphisms on reduced homology groups~$\widetilde{H}_{n-1}(\cdot;\RZ)$. We want to find a retraction, i.e.~an extension of~$f$. As~$\dim(X)=n$, we can apply Hopf's extension theorem (Theorem \ref{thm_hopf_extension}):~$f$ has an extension if and only if~$\ker i_*\subseteq \ker f_*$. As~$f_*$ is injective,~$f$ has an extension if and only if~$i_*$ is injective.

In the exact sequence (with the coefficient group~$\RZ$)
\begin{equation}\label{eq10}
H_n(Y)\to H_n(Y,Y\setminus \op{C})\to \widetilde{H}_{n-1}(Y\setminus \op{C})
\end{equation}
the first homomorphism is trivial by assumption (note that~$H_n(X/A)\cong H_n(X,A)$ as~$n\geq 1$), so the second homomorphism is injective. The latter is~$i_*$ up to isomorphism, because one has natural isomorphisms
\begin{equation*}
H_n(Y,Y\setminus \op{C})\cong H_n(C,\bd{C})\cong \widetilde{H}_{n-1}(\bd{C}).
\end{equation*}
As a result,~$i_*$ is injective. It implies that~$f$ has an extension, which is the sought retraction.

The particular case~$n=1$ or~$2$ corresponds to the particular case in Hopf's extension theorem.
\end{proof}

It also implies that Theorem \ref{thm_graph} extends from graphs to simplicial complexes of dimension at most~$3$ that are not necessarily pure.
\begin{corollary}\label{cor_dim3}
Let~$(X,A)$ be an almost Euclidean compact pair, where~$X$ and~$A$ are ANRs and~$\dim(X)\leq 3$. The following statements are equivalent:
\begin{itemize}
\item The pair~$\C{X,A}$ has the surjection property,
\item Every regular cell~$C\subseteq X\setminus A$ belongs to a relative cycle.
\end{itemize}
\end{corollary}
\begin{proof}
For every regular~$n$-cell~$C$, one has~$n=\dim(X)$ or~$n=1$ or~$n=2$ so we can apply Corollary \ref{cor_retraction}. If~$C$ does not belong to a relative cycle, then there exists a retraction~$r:X\setminus \op{C}\to\bd{C}$, so the quotient map~$q_C:(X,A)\to (\Sp_n,s)$ is null-homotopic (Theorem \ref{thm_retraction_quotient}).
\end{proof}

We will see in Section \ref{sec_counter_examples} that the dimension assumption in Corollaries \ref{cor_equivalence}, \ref{cor_retraction} and \ref{cor_dim3} cannot be dropped.

We also obtain a relationship between cycles and the~$\epsilon$-surjection property for ANRs.
\begin{theorem}[Cycles vs $\epsilon$-surjection property]\label{thm_cycle_epsilon}
Let~$(X,A)$ be an almost Euclidean compact pair, where~$X$ is an ANR. If every regular cell~$C\subseteq X\setminus A$ belongs to a relative cycle, then~$(X,A)$ has the~$\epsilon$-surjection property for some~$\epsilon>0$. 
\end{theorem}
\begin{proof}
As~$X$ is a compact ANR, if~$\epsilon$ is sufficiently small, then every function~$f:X\to X$ satisfying~$f|_A=\id_A$ and~$d(f,\id_X)<\epsilon$ is homotopic to~$\id_X$, via a homotopy~$h_t:X\to X$ such that~$h_t|_A=\id_A$. We claim that~$(X,A)$ has the~$\epsilon$-surjection property. Assume otherwise, let~$f:X\to X$ be non-surjective continuous, satisfy~$f|_A=\id_A$ and~$d(f,\id_X)<\epsilon$. $f$ is homotopic to~$\id_X:(X,A)\to (X,A)$.

Let~$C\subseteq X\setminus A$ be a regular~$n$-cell which is disjoint from~$\im f$,~$i:(X\setminus \op{C},A)\to (X,A)$ the inclusion map and let~$\widetilde{f}:(X,A)\to (X\setminus \op{C},A)$ be such that~$f=i\circ \widetilde{f}$.

As~$f=i\circ \widetilde{f}$ is homotopic to~$\id_X$,~$i_*\circ \widetilde{f}_*$ is an isomorphism on homology groups of~$(X,A)$, so~$i_*:H_n(X\setminus \op{C},A)\to H_n(X,A)$ is surjective (coefficients are implicitly in~$\RZ$). In the exact sequence
\begin{equation*}
H_n(X\setminus \op{C},A)\stackrel{i_*}{\longrightarrow} H_n(X,A)\longrightarrow H_n(X,X\setminus \op{C})
\end{equation*}
the second homomorphism is therefore trivial, contradicting the assumption that~$C$ belongs to a relative cycle. Therefore, there is no such function~$f$, so~$(X,A)$ has the~$\epsilon$-surjection property.
\end{proof}

The converse implication in Theorem \ref{thm_cycle_epsilon} does not hold, even under dimension assumptions., as the next example shows.
\begin{example}[Bing's house]
Let~$X$ be Bing's house, or the house with two rooms, which is a pure~$2$-dimensional simplicial complex, and~$A=\emptyset$. It was  proved in \cite{AH22} that it has the~$\epsilon$-surjection property, by using the fact that the star of each vertex in~$X$ is the cone of a graph which consists of a union of cycles. However,~$X$ is contractible, so no~$2$-cell belongs to a cycle.
\end{example}

\section{Counter-examples}\label{sec_counter_examples}
We consider a family of spaces which enables us to find counter-examples, and show that certain assumptions in the previous results cannot be dropped.

\subsection{A family of spaces}

Let~$m,n$ be two natural numbers and let~$f:\Sp_m\to\Sp_n$ be continuous. One can attach~$\BB_{m+1}$ to~$\BB_{n+1}$ along their boundaries using~$f$. We obtain the space~$X_f=\BB_{n+1}\cup_f\BB_{m+1}$, which is the quotient of~$\BB_{m+1}\sqcup\BB_{n+1}$ obtained by identifying~$x\in\Sp_{m+1}=\partial\BB_{m+1}$ to~$f(x)\in\Sp_n=\partial\BB_n$. It is illustrated in Figure \ref{fig_Xf}.
\begin{figure}[h]
\centering
\includegraphics{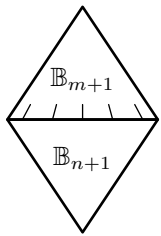}
\caption{The space~$X_f$. The ball~$\BB_{n+1}$ is visualized as the cone~$\C{\Sp_n}$, the center of the ball is the tip of the cone.}\label{fig_Xf}
\end{figure}

We will see that the properties of~$X_f$ we are interested in only depend on the homotopy class of~$f$. Therefore, by the Simplicial Approximation Theorem (Theorem 2C.1 in \cite{Hatcher02}), we can always assume that~$f$ is a simplicial map, implying that~$X_f$ is a finite simplicial complex.

We will need to consider the suspension construct. We recall that the suspension of a space~$X$ is~$\sus X=X*\Sp_0$. One has~$\sus \Sp_n=\Sp_{n+1}$. To a map~$f:X\to Y$ is associated its suspension~$\sus f:\sus X\to \sus Y$, which is naturally defined from~$f$. Note that for~$f:\Sp_m\to\Sp_n$, one has~$\sus f:\Sp_{m+1}\to\Sp_{n+1}$.

\subsection{Cycles vs the surjection property}
We show that the equivalences stated in Corollaries \ref{cor_equivalence} and \ref{cor_retraction} are no more valid when the assumptions about the dimension of the space and the dimension of the regular cells are dropped. For clarity, let us summarize what we have so far.

For any compact pair~$(X,A)$, if~$C\subseteq X\setminus A$ is a regular~$n$-cell, then we have the following implications~$1.\Rightarrow 2. \Rightarrow 3.$:
\begin{enumerate}
\item There is a retraction~$r:X\setminus \op{C}\to\bd{C}$ which is constant on~$A$,
\item The quotient map~$q_C:(X,A)\to (\Sp_n,s)$ is null-homotopic,
\item $C$ does not belong to a relative cycle.
\end{enumerate}

When~$(X,A)$ is a finite simplicial pair and~$\dim(X)=n$, all these conditions are equivalent by Corollaries \ref{cor_equivalence} and \ref{cor_retraction}. However we show that if~$\dim(X)>n$, then both implications~$2.\Rightarrow 1.$ and~$3.\Rightarrow 2.$ fail in general (still assuming that~$(X,A)$ is a finite simplicial pair).


Note that~$\BB_{n+1}$ is a regular~$(n+1)$-cell in~$X_f$.  We relate the conditions 1., 2.~and 3.~for the space~$X_f$ and the regular~$(n+1)$-cell~$C=\BB_{n+1}$ in terms of the properties of~$f$.

\begin{theorem}\label{thm_Xf}
Let~$m, n\in\N$ and~$f:\Sp_m\to\Sp_n$ be continuous. The space~$X_f=\BB_{n+1}\cup_f\BB_{m+1}$ enjoys the following properties:
\begin{enumerate}[label=(\alph*)]
\item There exists a retraction~$r:\Sp_n\cup_f\BB_{m+1}\to \Sp_n$ iff~$f$ is null-homotopic,
\item The quotient map~$q:X_f\to \Sp_{n+1}$ is null-homotopic iff~$\sus f$ is null-homotopic,
\item If~$m\neq n$, then~$\BB_{n+1}$ does not belong to a cycle in~$X_f$.
\end{enumerate}

Moreover,~$(\C{X_f},X_f)$ has the surjection property iff~$\sus f$ is not null-homotopic.
\end{theorem}

\begin{proof}
(a) Note that~$\Sp_n\cup_f\BB_{m+1}$ is the quotient of~$\BB_{m+1}$ by the equivalence relation~$x\sim y$ iff~$x=y$, or~$x,y\in\Sp_m$ and~$f(x)=f(y)$. Therefore, there is a one-to-one correspondence between the continuous functions~$r:\Sp_n\cup_f\BB_{m+1}\to\Sp_n$ and the continuous functions~$F:\BB_{m+1}\to\Sp_n$ that respect~$\sim$. A function~$r:\Sp_n\cup_f\BB_{m+1}\to\Sp_n$ is a retraction iff the corresponding function~$F:\BB_{m+1}\to\Sp_n$ is an extension of~$f$, i.e.~is a null-homotopy of~$f$.

(b) As~$\BB_{n+1}$ is contractible, the quotient map~$p:X_f\to X_f/\BB_{n+1}\cong\Sp_{m+1}$ is a homotopy equivalence. We show that the homotopy class of~$q:X_f\to\Sp_{n+1}$ is sent to the homotopy class of~$\sus f:\Sp_{m+1}\to\Sp_{n+1}$ under the equivalence, which implies that~$q$ is null-homotopic if and only if~$\sus f$ is. More precisely, we show that~$q\sim\sus f \circ p$ (these two maps are illustrated on Figures \ref{fig_q} and \ref{fig_sigmaf_p}).

\begin{figure}[h]
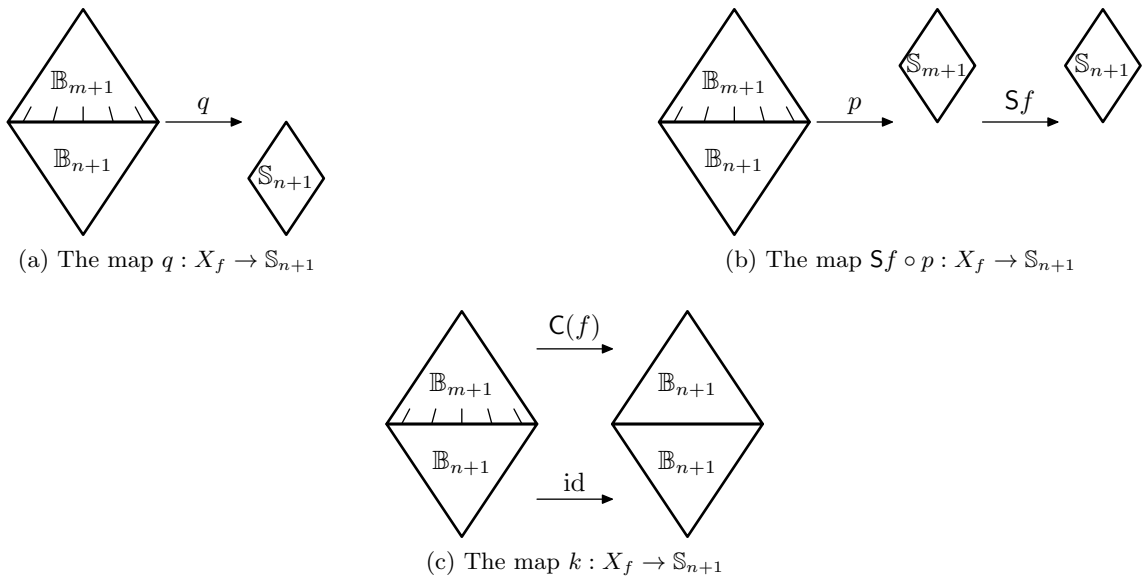

\centering
\subfloat[The map $q:X_f\to\Sp_{n+1}$]{\includegraphics{space-2}\label{fig_q}}
\hspace{\stretch{1}}
\subfloat[The map $\sus f\circ p:X_f\to\Sp_{n+1}$]{\includegraphics{space-1}\label{fig_sigmaf_p}}
\hspace{\stretch{1}}
\subfloat[The map $k:X_f\to\Sp_{n+1}$]{\includegraphics{space-0}\label{fig_k}}
\caption{Illustration of the proof of (b) in Theorem \ref{thm_Xf}}
\end{figure}

The sphere~$\Sp_{n+1}$ is a union of two balls~$\BB_{n+1}$, called the lower and upper hemispheres, joined at the equator. Let~$j:\BB_{n+1}\sqcup\BB_{m+1}\to\Sp_{n+1}$ send~$\BB_{n+1}$ homeomorphically to the lower hemisphere of~$\Sp_{n+1}$, and send~$\BB_{m+1}$ to the upper hemisphere of~$\Sp_{n+1}$ by applying the map~$\C{f}:\C{\Sp_m}\to\C{\Sp_n}$, i.e.~$\C{f}:\BB_{m+1}\to\BB_{n+1}$. The space~$X_f$ is a quotient of~$\BB_{m+1}\sqcup\BB_{n+1}$ and the map~$j$ respects that quotient: if~$x,y\in\Sp_m=\partial \BB_{m+1}$ and~$z\in\Sp_n=\partial \BB_{n+1}$ are such that~$f(x)=f(y)=z$, then~$j(x)=j(y)=j(z)$ is~$z$, seen as a point of the equator of~$\Sp_{n+1}$.

Therefore,~$j$ induces a continuous function~$k:X_f\to\Sp_{n+1}$, illustrated on Figure \ref{fig_k}. The functions~$q$ and~$\sus f\circ p$ can be factored using~$k$ as follows. Let~$U$ and~$L$ be the upper and lower halves of~$\Sp_{n+1}$. The quotient of~$\Sp_{n+1}$ by~$U$ or~$L$ is again~$\Sp_{n+1}$. Let~$q_U,q_L:\Sp_{n+1}\to\Sp_{n+1}$ be the corresponding quotient maps. One has~$q=q_U\circ k$ and~$\sus f\circ p=q_L\circ k$. One easily sees that~$q_U\sim q_L$, so~$q\sim\sus f\circ p$.


(c) As~$\BB_{n+1}$ is contractible,~$X_f$ is homotopy equivalent to~$X_f/\BB_{n+1}$, which is homeomorphic to~$\Sp_{m+1}$. Therefore,~$H_{n+1}(X_f;G)\cong H_{n+1}(\Sp_{m+1};G)$ which is trivial as~$m\neq n$. As a result, any homomorphism defined on~$H_{n+1}(X_f;G)$ is trivial.

We prove the last assertion. For a regular~$(m+1)$-cell~$C\subseteq X_f$, the corresponding quotient map~$q_C:X_f\to\Sp_{m+1}$ is never null-homotopic, because it is a homotopy equivalence. Therefore,~$(\C{X_f},X_f)$ enjoys the surjection property if and only if~$\sus f$ is not null-homotopic.
\end{proof}


The first application is that in Corollary \ref{cor_equivalence}, one cannot drop the assumption that all the cells have the same dimension as the space.
\begin{example}
Let~$h:\Sp_3\to\Sp_2$ be the Hopf map. It is known that~$h$ and~$\sus h$ are not null-homotopic (Corollary 4J.4 in \cite{Hatcher02}). Therefore, in~$X_h$ the quotient map is not null-homotopic and~$(\C{X_h},X_h)$ enjoys the surjection property although no regular cell belongs to a cycle, so the implication~$3.\Rightarrow 2.$ fails for~$X_h$.
\end{example}

A second application is that the dimension assumption in Corollary \ref{cor_retraction} is needed.

\begin{example}
Let again~$h:\Sp_3\to\Sp_2$ be the Hopf map. The function~$2h$ (which is the concatenation of~$h$ with itself using the homotopy group operation) is not null-homotopic, however its suspension~$\sus 2h$ is (again by Corollary 4J.4 in \cite{Hatcher02}). Therefore, in~$X_{2h}$ the quotient is null-homotopic and~$(\C{X_{2h}},X_{2h})$ does not satisfy the surjection property, but there is no retraction, so the implication $2.\Rightarrow 1.$ fails for~$X_{2h}$.
\end{example}

\subsection{Product}
We have shown with Proposition \ref{prop_product} that for finite simplicial pairs, if~$(X,A)\times (Y,B)$ has the~$\epsilon$-surjection property for some~$\epsilon>0$, then both~$(X,A)$ and~$(Y,B)$ satisfy the~$\delta$-surjection property for some~$\delta>0$.

We prove that the converse implication does not hold.
\begin{theorem}\label{thm_product_counterexample}
There exists a finite simplicial pair~$(X,A)$ that has the surjection property, but the product~$(X,A)\times (\BB_1,\Sp_0)$ does not have the~$\epsilon$-surjection property for any~$\epsilon>0$.

There exists a finite simplicial complex~$Y$ that has the~$\epsilon$-surjection property for some~$\epsilon>0$, but the product~$Y\times \Sp_1$ does not have the~$\delta$-surjection property for any~$\delta>0$.
\end{theorem}

\begin{proof}
We choose~$m,n\in\N$ and~$f:\Sp_m\to\Sp_n$ so that~$\sus f$ is not-null-homotopic but~$\sus^2 f$ is (we will give detailed references in Section \ref{sec_suspension} below on the existence of such functions). We then let~$(X,A)=(\C{X_f},X_f)$ and~$Y=\sus X_f$.

\begin{lemma}
$\sus X_f$ is homeomorphic to~$X_{\sus f}$.
\end{lemma}
\begin{proof}
Both spaces are the quotients of~$(\BB_{n+1}\sqcup\BB_{m+1})\times [0,1]$ by the equivalence relation generated by:
\begin{align}
(y,t)\sim (f(y),t)&\text{ for }y\in\Sp_m,\label{eq1}\\
(x,0)\sim (x',0)\text{ and }(x,1)\sim (x',1)&\text{ for }x,x'\in\BB_{n+1},\label{eq2}\\
(y,0)\sim(y',0)\text{ and }(y,1)\sim(y',1)&\text{ for }y,y'\in\BB_{m+1}.\label{eq3}
\end{align}

Taking the quotient by \eqref{eq1} first yields~$X_f\times [0,1]$, and then the quotient by \eqref{eq2}, \eqref{eq3} yields~$\sus X_f$. Taking the quotient by \eqref{eq2}, \eqref{eq3} first yields~$\sus\BB_{n+1}\sqcup\sus \BB_{m+1}$ and then the quotient by \eqref{eq1} yields~$X_{\sus f}$.
\end{proof}

Therefore,
\begin{align*}
(\C{X_f},X_f)\times (\C{\Sp_0},\Sp_0)&\cong (\C{X_f*\Sp_0},X_f*\Sp_0)\\
&\cong(\C{\sus X_f},\sus X_f)\\
&\cong(\C{X_{\sus f}},X_{\sus f}).
\end{align*}
If~$\sus f$ is not null-homotopic but~$\sus^2 f$ is, then~$(\C{X_f},X_f)$ enjoys the surjection property but~$(\C{X_{\sus f}},X_{\sus f})$ does not. For cone pairs, the surjection property is equivalent to the~$\epsilon$-surjection property for any~$\epsilon>0$ by Proposition \ref{prop_cone_epsilon}.

The space~$Y=\sus X_f$ is finitely conical with two open cones~$\OC{X_f}$, so~$Y$ has the~$\epsilon$-surjection property for some~$\epsilon>0$. The product~$Y\times\Sp_1=\sus X_f\times \sus \Sp_0$ is finitely conical with four open cones~$\OC{X_{\sus f}}$, so~$Y\times \Sp_1$ does not satisfy the~$\epsilon$-surjection property for any~$\epsilon>0$. In both case we use Theorem \ref{thm_local_surjection}.
\end{proof}

\subsubsection{Iteration of the suspension}\label{sec_suspension}
The existence of the function~$f$ in the proof of Theorem \ref{thm_product_counterexample} can be extracted from the literature on homotopy groups of spheres and the suspension homomorphism, as we explain now.
\begin{theorem}
For~$(m,n)=(7,3)$ or~$(8,4)$, there exists~$f:\Sp_m\to\Sp_n$ such that~$\sus f$ is not null-homotopic but~$\sus^2 f$ is.
\end{theorem}

The results in the literature are usually expressed in terms of the reduced suspension~$\Sigma$ rather than the suspension~$\sus$ (we use the notation~$\Sigma$ as in \cite{Hatcher02}, while the other references use the notation~$E$). We recall that~$\Sigma X$ is defined as the quotient of~$\sus X=X*\Sp_0$ by the segment~$\{x_0\}*\Sp_0$ for some~$x_0\in X$. However, for a CW-complex~$X$,~$\Sigma X$ and~$\sus X$ are homotopy equivalent (Example 0.10 in \cite{Hatcher02}) and if~$X,Y$ are CW-complexes and~$f:X\to Y$ is continuous, then~$\Sigma f$ is null-homotopic if and only if~$\sus X$ is. More precisely, the quotient map~$\varphi_X:\sus X\to\Sigma X$ is a homotopy equivalence and one has~$\Sigma f\circ \varphi_X=\varphi_Y\circ \sus f$.

We use the following homotopy groups of spheres, that can be found in Hatcher \cite{Hatcher02} (\S 4.1, p.~339):
\begin{equation*}
\pi_7(\Sp_3)\cong \Z_2\quad\pi_8(\Sp_4)\cong \Z_2^2\quad\pi_9(\Sp_5)\cong \Z_2\quad\pi_{10}(\Sp_6)\cong 0
\end{equation*}

\paragraph{Whitehead product.}
The Whitehead product is a way of combining an element~$f\in\pi_p(\Sp_n)$ with an element~$g\in\pi_q(\Sp_n)$ into their product~$[f,g]\in\pi_{p+q-1}(\Sp_n)$. The suspension of a Whitehead product is always null-homotopic (see Whitehead \cite{Whitehead46}, Theorem 3.11, p.~470).

\begin{theorem}[Suspension of Whitehead product]
If~$f\in\pi_p(\Sp_n)$ and~$g\in\pi_q(\Sp_n)$, then one has~$\Sigma [f,g]=0$.
\end{theorem}

The next result was proved by Toda \cite{Toda52} (\S5.ii, p.~79).
\begin{lemma}
There exists~$f:\Sp_7\to\Sp_3$ such that~$\Sigma f$ is not null-homotopic but~$\Sigma^2 f$ is.
\end{lemma}
\begin{proof}
Let~$i_4:\Sp_4\to\Sp_4$ be the identity. Let~$h_2:\Sp_3\to\Sp_2$ be the Hopf map and~$\nu_{4}=\Sigma^2h_2:\Sp_{5}\to\Sp_{4}$. It is proved in \cite{Toda52}, \S 5.ii that there exists~$f:\Sp_7\to\Sp_3$ such that~$\Sigma f=[\nu_4,i_4]\neq 0$, implying that~$\Sigma^2 f=0$. (It is defined as~$f=a_3\circ \nu_6$ for some particular~$a_3:\Sp_6\to\Sp_3$.)
\end{proof}

\paragraph{Suspension homomorphism.}
We will use the fact that the suspension homomorphism~$\Sigma:\pi_8(\Sp_4)\to \pi_9(\Sp_5)$ is surjective, which follows from the famous Freudenthal suspension theorem (see Whitehead \cite{Whitehead46} p.~468 for the statement).
\begin{theorem}[Freudenthal \cite{Freudenthal37}]
The suspension homomorphism
\begin{equation*}
\Sigma:\pi_{n+k}(\Sp_n)\to\pi_{n+k+1}(\Sp_{n+1})
\end{equation*}
is an isomorphism if~$k<n-1$, and is surjective if~$k=n-1$ or~$k=n$ is even.
\end{theorem}

\begin{lemma}
There exists~$f:\Sp_8\to\Sp_4$ such that~$\Sigma f$ is not null-homotopic but~$\Sigma^2 f$ is.
\end{lemma}
\begin{proof}

As~$\pi_9(\Sp_5)\cong \Z_2$ is non-trivial, let~$g\in\pi_9(\Sp_5)$ be non-zero. As~$\Sigma:\pi_8(\Sp_4)\to \pi_9(\Sp_5)$ is surjective, let~$f\in\pi_8(\Sp_4)$ be such that~$\Sigma f=g$. As~$\pi_{10}(\Sp_6)$ is trivial,~$\Sigma^2 f=0$.

Alternatively, Whitehead \cite{Whitehead50} (\S 9, p.~234) shows that there exists~$\gamma\in\pi_8(\Sp_4)$ such that~$\Sigma\gamma=\pm[i_5,i_5]\neq 0$, where~$i_5:\Sp_5\to\Sp_5$ is the identity, therefore~$\Sigma^2\gamma=0$. (It is defined as~$\gamma=J(\gamma_3)$ for some~$\gamma_3$ in \cite{Whitehead50}).
\end{proof}

\section{Applications to computable type}\label{sec_computable_type}

Computability theory provides notions of computability for subsets of~$\N$:
\begin{itemize}
\item A set~$A\subseteq\N$ is \textbf{computable} if there exists a program, or a Turing machine, that on input~$n\in\N$ decides in finite time whether~$n\in A$,
\item A set~$A\subseteq\N$ is \textbf{computably enumerable (c.e.)} if there exists a program or a Turing machine, that on input~$n\in\N$, halts if and only if~$n\in A$.
\end{itemize}
The halting set~$H\subseteq\N$, which is the set of indices of Turing machines that halt, is a famous example of a c.e.~set that is not computable, discovered by Turing in his seminal article \cite{Turing36}.

These notions immediately extend to subsets of countable sets which admit a computable bijection with~$\N$, such as the set of rational numbers, or the set~$\N^*$ of finite sequences of natural unmbers.

Such notions can then be used to define computability notions for other mathematical objects. We recall the notions of computable and semicomputable compact subsets of the Hilbert cube.

We say that an element of the Hilbert cube~$Q=[0,1]^\N$ is \textbf{finitary} if it is a sequence of rational numbers which is eventually~$0$. By standard arguments, there is a computable enumeration~$(B_i)_{i\in\N}$ of the metric balls centered at finitary elements, with positive rational radii.


\begin{definition}
A compact set~$K\subseteq Q$ is \textbf{semicomputable} if the set
\begin{equation*}
\{(i_1,\ldots,i_n)\in\N^*:K\subseteq B_{i_1}\cup \ldots\cup B_{i_n}\}
\end{equation*}
is computably enumerable.
\end{definition}

\begin{definition}
A compact set~$K\subseteq Q$ is \textbf{computable} if it is semicomputable and the set
\begin{equation*}
\{i\in\N:K\cap B_i\neq \emptyset\}
\end{equation*}
is computably enumerable.
\end{definition}

There exist computable sets that are not computable, by a simple encoding of the halting set~$H$. The simplest example is the interval~$[0,r]$ (embedded in~$Q$ as~$[0,r]\times Q$), where~$r=1-\sum_{i\in H}2^{-i-1}$. This example is an interval, but similar constructions can be made with disks,  and more generally~$n$-dimensional balls. It is a striking result that no such construction is possible with spheres. This result is due to Miller \cite{Miller02}, and later motivated the introduction of the notion of computable type by Iljazovi\'c.

Say that a pair~$(Y,B)$ is homeomorphic to, or is a copy of the pair~$(X,A)$ if there exists a homeomorphism~$f:X\to Y$ such that~$f(A)=B$.
\begin{definition}
A compact metrizable space~$X$ has \textbf{computable type} if for every copy~$Y\subseteq Q$ of~$X$, if~$Y$ is semicomputable  then~$Y$ is computable.

A compact pair~$(X,A)$ has \textbf{computable type} if for every copy~$(Y,B)$ of~$(X,A)$ in the Hilbert cube~$Q$, if~$Y$ and~$B$ are semicomputable, then~$Y$ is computable.
\end{definition}

The first examples of spaces and pairs having computable type were given by Miller \cite{Miller02}, who proved that spheres~$\Sp_n$ and balls with their bounding spheres~$(\BB_{n+1},\Sp_n)$ have computable type. Iljazovi\'c \cite{Iljazovic13} proved more generally that closed manifolds~$M$ and compact manifolds with boundary~$(M,\partial M)$ have computable type. Many more examples were studied in \cite{2009chainable_and_circularly_chainable,BurnikI14,2018manifolds,IP18,CickovicIV19,2020graphs,2020pseudocubes,CI21}. For instance, the Warsaw circle and the Warsaw disk with its bounding Warsaw circle have computable type \cite{2009chainable_and_circularly_chainable,IP18}.

In \cite{AH22}, we have studied the case of finite simplicial complexes and proved that a finite simplicial pair~$(X,A)$ has computable type if and only if it has the~$\epsilon$-surjection property for some~$\epsilon>0$. This result implies for instance that Bing's house has computable type, while the dunce hat does not. The proof of the result more generally works on compact ANRs that are covered by finitely many regular cells.

The results obtained in the present article have immediate applications to the computable type property.

\subsection{Decidability}
The reduction of computable type to homology implies a first decidability result in restricted cases.
\begin{proposition}
For finite simplicial pairs~$(X,A)$ such that~$X$ is pure or has dimension at most~$4$, whether~$(X,A)$ has computable type is decidable.
\end{proposition}
\begin{proof}
Such a pair~$(X,A)$ is finitely conical and has computable type if and only if each cone pair~$\C{L_i,N_i}$ has the surjection property. If~$X$ is pure or has dimension at most~$4$, the each~$L_i$ is pure or has dimension at most~$3$, so the surjection property for~$\C{L_i,N_i}$ is equivalent to the fact that each maximal simplex of~$L_i$ belongs to a relative cycle in~$(L_i,N_i)$ by Corollaries \ref{cor_equivalence} and \ref{cor_dim3}. As homology groups of finite simplicial complexes can be computed, this property is therefore decidable.
\end{proof}

A recent result by Filakovsk{\'{y}} and Vokr{\'{\i}}nek even implies that computable type is decidable for \emph{all} finite simplicial pairs.
\begin{theorem}[\cite{FV20}]\label{thm_decidable_homotopy}
There is an algorithm that decides the existence of a pointed homotopy between given simplicial maps~$f,g:X\to Y$, where~$X,Y$ are finite simplicial sets and~$Y$ is simply connected.
\end{theorem}
First, it implies that the null-homotopy of simplicial maps from finite simplicial pairs to spheres is decidable.
\begin{corollary}\label{cor_dec_null_hom}
There is an algorithm that decides the null-homotopy of a simplicial map~$f:(X,A)\to (\Sp_n,s)$, where~$(X,A)$ is a finite simplicial pair and~$n\in\N$.
\end{corollary}
\begin{proof}
When~$n\leq 1$, the null-homotopy of~$f$ is equivalent to the triviality of the homomorphism~$f_*:\tilde{H}_1(X,A;\RZ)\to \tilde{H}_1(\Sp_1,s;\RZ)$ by Hopf's classification theorem (Theorem \ref{thm_hopf_classification}), which is decidable. For~$n\geq 2$,~$\Sp_n$ is simply connected so we can apply Theorem \ref{thm_decidable_homotopy}, observing that the null-homotopy of~$f$ is equivalent to the null-homotopy of~$\tilde{f}:(X/A,A/A)\to (\Sp_n,s)$, which is a pointed homotopy.
\end{proof}

Together with our previous results, it implies that for finite simplicial pairs~$(X,A)$, whether~$(X,A)$ has computable type is decidable.
\begin{corollary}\label{cor_dec_comp_type}
There is an algorithm that decides whether a finite simplicial pair~$(X,A)$ has computable type.
\end{corollary}
\begin{proof}
One has the following chain of equivalences:
\begin{align*}
&(X,A)\text{ has computable type}\\
&\iff \exists \epsilon>0, (X,A)\text{ has the~$\epsilon$-surjection property (from \cite{AH22})}\\
&\iff \text{all the cone pairs~$\C{L_i,N_i}$ have the surjection property (Theorem \ref{thm_local_surjection})}\\
&\iff\text{all the quotient maps~$q:\C{L_i,N_i}\to(\Sp_n,s)$ are not null-homotopic (Theorem \ref{thm_surj_prop_cone})}.
\end{align*}
The quotient maps are associated to the maximal simplices of~$L_i$ that are not contained in~$N_i$. There are finitely many such simplices and finitely many cone pairs~$\C{L_i,N_i}$, so this predicate is decidable by Corollary \ref{cor_dec_null_hom}.
\end{proof}

\subsection{Cones of manifolds}

Let~$(M,\partial M)$ be a compact~$n$-manifold with (possible empty) boundary. It is almost~$n$-Euclidean, because every point of~$M\setminus \partial M$ is~$n$-Euclidean.

One has~$H_n(M,\partial M;\Z/2\Z)\cong \Z/2\Z$ and the fundamental homology class intuitively contains every point of~$M\setminus \partial M$. More formally, every regular~$n$-cell~$C\subseteq M\setminus \partial M$ belongs to a relative cycle with coefficients in~$\Z/2\Z$ (therefore in~$\RZ$), in the sense of Definition \ref{def_cycle}. Indeed,~$M\setminus \partial M$ is~$\Z/2\Z$-orientable so the homomorphism
\begin{equation*}
H_n(M,\partial M;\Z/2\Z)\to H_n(M,M\setminus \op{C};\Z/2\Z)\cong \Z/2\Z
\end{equation*}
is an isomorphism, sending the fundamental class to the generator of~$\Z/2\Z$ (see~\cite{Hatcher02} for instance).

Therefore, we can apply Corollary \ref{cor_implication}, giving the next result.
\begin{corollary}\label{cor_cone_manifold}
If~$(M,\partial M)$ is a compact manifold with possibly empty boundary, then the pair~$\C{M,\partial M}$ has the surjection property hence has computable type.
\end{corollary}

Strictly speaking, the results in \cite{AH22} are stated for finite simplicial complexes. However, the arguments work for any space which is covered by finitely many regular cells, including cones of compact manifolds.

\subsection{The odd subcomplex}
If~$K$ is a finite simplicial complex, then there is a natural subcomplex~$A$ of~$K$ such that the pair~$(K,A)$ has the~$\epsilon$-surjection property for some~$\epsilon>0$. This subcomplex is the \emph{odd subcomplex} of~$K$.

\begin{definition}
Let~$K$ be a finite simplicial complex of dimension~$m$. For~$0\leq n< m$,
the \textbf{odd~$n$-subcomplex}~$\oddn nK$ of~$K$ is the collection
of all~$n$-simplices of~$K$ that are contained in an odd number
of maximal~$(n+1)$-simplices in~$K$. The \textbf{odd subcomplex}~$\odd K$
of~$K$ is the union~$\bigcup_{0\leq n< m}\oddn nK$.
\end{definition}

Clearly, for every finite simplicial complex~$K$,~$\odd K$ has empty interior because it contains no maximal simplex of~$K$. The next result is announced in \cite{AH22}, without proof.

\begin{theorem}
\label{thm_odd}For every finite simplicial complex~$K$, the pair~$(K,\odd K)$ has the~$\epsilon$-surjection property for some~$\epsilon>0$, hence has computable type.
\end{theorem}

\begin{proof}
Any finite simplicial complex is an ANR, so the pair~$(K,\odd K)$ has the assumption of Theorem \ref{thm_cycle_epsilon}, therefore it is sufficient to show that every maximal simplex in~$K$ belongs to a relative cycle. We use the formulation of this latter property provided by Proposition \ref{prop_cycle_simplicial}.

Let~$S$ be a maximal~$n$-simplex of~$K$. Let~$c$ be the formal sum of the maximal~$n$-simplices of~$K$, seen as an~$n$-chain with coefficients in~$\Z/2\Z$. The boundary of~$c$ is the sum of the~$(n-1)$-simplices that are contained in an odd number of maximal~$n$-simplices. Therefore,~$\partial c$ is a chain of~$\odd K$ so~$c$ is a relative cycle in~$(K,\odd K)$. As~$S$ is itself a maximal~$n$-simplex, its coefficient in~$c$ is~$1$, so~$S$ belongs to a relative cycle and the proof is complete.
\end{proof}
\subsection{Computable type and the product}
Whether the computable type property is preserved by taking products was an open question, raised by \v{C}elar and Iljazovi\'{c} in \cite{CI21}. Theorem \ref{thm_product_counterexample} enables us to give a negative answer  to this question. Note that~$(\BB_1,\Sp_0)$ and~$\Sp_1$ both have computable type.

\begin{corollary}\label{cor_ct_product}
There exists a finite simplicial pair~$(X,A)$ that has computable type, but such that the product~$(X,A)\times (\BB_1,\Sp_0)$ does not.

There exists a finite simplicial complex~$Y$ that has computable type, but such that the product~$Y\times \Sp_1$ does not.
\end{corollary}
\begin{proof}
For simplicial spaces and pairs, computable type is equivalent to the~$\epsilon$-surjection property for some~$\epsilon$, so we can simply apply Theorem \ref{thm_product_counterexample}.
\end{proof}

Note however that the converse direction holds: it was proved in \cite{AH22c} that if~$(X,A)\times (Y,B)$ has computable type, then both~$(X,A)$ and~$(Y,B)$ have computable type (assuming that they both have a semicomputable copy, which is the case for finite simplicial pairs). This implication is consistent with Proposition \ref{prop_product}, which shows that the~$\epsilon$-surjection property behaves similarly.

\section{Conclusion}\label{sec_conclusion}

The surjection property and the~$\epsilon$-surjection property were introduced in \cite{AH22} as topological characterizations of a computability-theoretic property of certain compact metrizable spaces, called computable type. In this article, we have developed several techniques to establish or refute these properties, applicable to classes of spaces including finite simplicial complexes and compact manifolds.

We have established precise relationships between the ($\epsilon$-)surjection property and the homotopy of certain quotient maps to spheres, which enables us to take advantage of results from homology and homotopy theory. We have given applications of these techniques. The first one is that the cone of a compact manifold has computable type. The second one is that any finite simplicial complex together with its odd boundary has computable type. The third and most important application is an answer to a question raised in \cite{CI21}, showing that the computable type property is not preserved by taking products.

%


\bibliographystyle{elsarticle-num} 
\bibliography{biblio}

\newpage
\appendix

We use the appendix to include statements and proofs of results that are either classical but phrased differently in most references, or that are folklore results for which we found no reference.

\section{Homotopies}\label{sec_hom}

\subsection{Pair vs quotient}

In many cases, a pair~$(X,A)$ can be equivalently replaced by the quotient space~$X/A$. The following proposition is an instance of this fact, and is a combination of Proposition 4A.2 and Example 4A.3 in \cite{Hatcher02}.

Let~$(X,A)$ be a pair. There is a one-to-one correspondence between the continuous functions~$f:X\to \Sp_n$ that are constant on~$A$ and the continuous functions~$g:X/A\to\Sp_n$. If~$f:X\to \Sp_n$ is constant on~$A$ then we denote by~$\widetilde{f}:X/A\to\Sp_n$ the function satisfying~$f=\widetilde{f}\circ q$, where~$q:X\to X/A$ is the quotient map.

\begin{proposition}\label{prop_quotient_pair}
Let~$n\geq 1$ and let~$s$ be a distinguished point of~$\Sp_n$. A function of pairs~$f:(X,A)\to (\Sp_n,s)$ is null-homotopic if and only if~$\widetilde{f}:X/A\to\Sp_n$ is null-homotopic.
\end{proposition}
\begin{proof}
Let~$p$ be the equivalence class of~$A$ in~$X/A$. Note that~$\widetilde{f}$ can also be seen as a function of pairs~$\widetilde{f}_2:(X/A,p)\to (\Sp_n,s)$, which we denote differently to avoid confusions. The null-homotopy of~$f$ is easily equivalent to the null-homotopy of~$\widetilde{f}_2$, inducing a null-homotopy of~$\widetilde{f}$. We need to show that a null-homotopy of~$\widetilde{f}$ (a function between sets) implies a null-homotopy of~$\widetilde{f}_2$ (a function between pairs).

Let~$h_t:X/A\to\Sp_n$ be a homotopy from~$h_0=\widetilde{f}$ to a constant function~$h_1$. We use a different argument for~$n=1$ and for~$n\neq 1$.

Assume that~$n=1$. The circle~$\Sp_1$ can be seen as the additive group~$\R/\Z$ with~$s$ as the~$0$ element. Let~$g_t:(X/A,p)\to (\Sp_1,s)$ be defined by~$g_t(x)=h_t(x)-h_t(p)+h_0(p)$. One easily checks that~$g_0=h_0=\widetilde{f}$,~$g_t(p)=h_0(p)=s$ for all~$t$ and~$g_1(x)=s$ for all~$x$, so~$g_t$ is a null-homotopy of~$\widetilde{f}_2$.

Assume that~$n\neq 1$. Let~$Y=[0,1]\times (X/A)$ and~$B=(\{0,1\}\times X/A)\cup ([0,1]\times\{p\})\subseteq Y$. The null-homotopy~$h_t$ is a function~$h:Y\to\Sp_n$. Let~$g:B\to \Sp_n$ be the continuous function defined by
\begin{equation*}
g(t,x)=\begin{cases}
\widetilde{f}(x)&\text{if }t=0,\\
s&\text{otherwise.}
\end{cases}
\end{equation*}

A null-homotopy of~$\widetilde{f}_2$ is a continuous extension~$G:Y\to \Sp_n$ of~$g$. In order to show that such an extension exists, it is sufficient to show that~$g$ is homotopic to the restriction of~$h$ to~$B$, by applying Borsuk's homotopy extension theorem (Theorem 1.4.2, p.~38 in \cite{Mill01}) and the fact that~$\Sp_n$ is an ANR.

So let us define a homotopy~$k_t:B\to \Sp_n$ from~$k_0=g$ to~$k_1=h|_B$. We decompose~$B$ as~$B_0\cup B_1$, where~$B_0=\{0,1\}\times X/A$ and~$B_1=[0,1]\times \{p\}$. On~$B_0$,~$g$ and~$h$ coincide, and we define~$k_t(z)=g(z)=h(z)$ for all~$t$ and all~$z\in B_0$. On~$B_1$,~$g$ has constant value~$s$ and the restriction of~$h$ to~$B_1$ is a loop from~$s$ to itself. As~$\Sp_n$ is simply connected, that loop is contractible, i.e.~there exists a homotopy from~$g|_{B_1}$ to~$h|_{B_1}$, and we define~$k_t|_{B_1}$ as this homotopy.

The homotopy~$k_t:B\to \Sp_n$ is well-defined because it is consistent on~$B_0\cap B_1=\{0,1\}\times \{p\}$, where~$k_t(0,p)=k_t(1,p)=s$. Therefore,~$k_t$ is continuous and is a homotopy from~$g$ to~$h|_B$. It can be extended to a homotopy from some extension of~$g$ to~$h$, because~$\Sp_n$ is an ANR. The extension of~$g$ is a null-homotopy of~$\widetilde{f}_2$.
\end{proof}

\subsection{Homotopy extension property}
Let us recall an important property of pairs.
\begin{definition}
A pair~$(X,A)$ has the \textbf{homotopy extension property} if for every topological space~$Y$, for every continuous function~$f:X\to Y$ and every homotopy~$h_t:A\to Y$ such that~$h_0=f|_A$, there exists a homotopy~$H_t:X\to Y$ such that~$H_0=f$ and~$H_t|_A=h_t$.
\end{definition}
It is well-known that a pair~$(X,A)$ has the homotopy extension property if and only if~$X\times [0,1]$ retracts to~$(X\times\{0\})\cup (A\times [0,1])$ (see Hatcher \cite{Hatcher02}, Proposition A.18, p.~533). If~$(X,A)$ is a CW pair, then it has the homotopy extension property (see Hatcher \cite{Hatcher02}, Proposition 0.16, p.~15).
\begin{definition}
Let~$X$ be a topological space and~$(Y,d)$ a metric space. For~$\epsilon>0$, an~$\epsilon$\textbf{-homotopy} is a homotopy function~$h_t:X\to Y$ such that~$d(h_t(x),h_0(x))<\epsilon$ for all~$x,t$.
\end{definition}

The homotopy extension property actually implies a controlled version.
\begin{lemma}\label{lem_controlled_hep}
Let~$(X,A)$ be a pair satisfying the homotopy extension property. If~$(Y,d)$ is a metric space,~$f:X\to Y$ is continuous,~$h_t:A\to Y$ is an~$\epsilon$-homotopy such that~$h_0=f|_A$, then there exists an~$\epsilon$-homotopy~$H_t:X\to Y$ such that~$H_0=f$ and~$H_t|_A=h_t$.
\end{lemma}
A similar result with a similar proof holds for arbitrary compact pairs~$(X,A)$ when~$Y$ is an ANR (Theorem 4.1.3, p.~265 in \cite{Mill01}).
\begin{proof}
As~$(X,A)$ has the homotopy extension property, there exists a retraction~$r:X\times [0,1]\to (X\times \{0\})\cup (A\times [0,1])$ and a homotopy~$H_t:X\to Y$ extending~$H_0$ and~$h_t$. We define a retraction~$r':X\times [0,1]\to (X\times \{0\})\cup (A\times [0,1])$ so that~$H'_t(x)=H\circ r'(x,t)$ is an~$\epsilon$-homotopy.

Let~$U=\{x\in X:\forall t\in[0,1],d(H_t(x),H_0(x))<\epsilon\}$. As~$[0,1]$ is compact,~$U$ is an open set. Moreover,~$U$ contains~$A$. Let~$\delta:X\to [0,1]$ be a continuous function such that~$\delta(x)=1$ for~$x\in A$ and~$\delta(x)=0$ for~$x\in X\setminus U$. We define~$r'$ as follows:
\begin{equation*}
r'(x,t)=r(x,t\delta(x)).
\end{equation*}

First,~$r'$ is a retraction:~$r'(x,0)=r(x,0)=(x,0)$ and for~$x\in A$,~$r'(x,t)=r(x,t)=(x,t)$. Therefore,~$H'$ is a homotopy extending~$H_0$ and~$h_t$. We claim that~$H'$ is an~$\epsilon$-homotopy, which means that~$d(H'_t(x),H_0(x))<\epsilon$. If~$x\notin U$, then~$H\circ r'(x,t)=H\circ r(x,0)=H(x,0)=H_0(x)$. If~$x\in U$, then~$H\circ r'(x,t)=H\circ r(x,t\delta(x))$ is~$\epsilon$-close to~$H_0(x)$.
\end{proof}

\subsection{Absolute neighborhood retracts (ANRs)}

Absolute Neighborhood Retracts (ANRs) were introduced by Borsuk \cite{Borsuk32}
and play an eminent role in algebraic topology. They have many interesting
properties that we will use in our proofs, we recall them here. 
\begin{definition}
Let~$(X,A)$ be a pair. A \textbf{retraction}~$r:X\to A$
is a continuous function such that~$r|_{A}=\id_{A}$. If a retraction
exists, then we say that~$A$ is a \textbf{retract} of~$X$. 
\end{definition}

\begin{definition}
A compact Hausdorff space~$X$ is an \textbf{absolute neighborhood retract (ANR)} if every copy~$X_0$ of~$X$ in~$Q$ is a retract of an open set containing~$X_0$.
\end{definition}

%

The next theorem is a direct consequence of the main result in \cite{Whitehead48}.
\begin{theorem}\label{prop_quotient_anr}
If~$X$ and~$A\subseteq X$ are compact ANRs, then~$(X,A)$ has the homotopy extension property and~$X/A$ is an ANR.
\end{theorem}
%
%
%
%
%
%

The next lemma is an application of Borsuk's homotopy extension theorem that can be found as Theorem V.3.1 in \cite{Borsuk67} or as Exercice 4.1.5 in \cite{Mill01}.

\begin{lemma} \label{lem_ANR_extension}
Let~$Y$ be a compact metric ANR. For every~$\epsilon>0$ there exists~$\delta>0$ such that for all pairs~$(X,A)$ and all continuous functions~$f,g:A\to Y$ with~$d_{A}(f,g)<\delta$, if~$f$ has a continuous extension~$F:X\to Y$ then~$g$ has a continuous extension~$G:X\to Y$ such that~$d_{X}(F,G)<\epsilon$.
\end{lemma}

The following result is Theorem 4.1.1 in \cite{Mill01}.
\begin{lemma}\label{lem_close_homotopic}
Let~$Y$ be a compact metric ANR. For every~$\epsilon>0$ there exists~$\delta>0$ such that for all spaces~$X$, all continuous functions~$f,g:X\to Y$ satisfying~$d_X(f,g)<\delta$ are~$\epsilon$-homotopic.
\end{lemma}


\section{Cone and join}\label{sec_cone_join}
The cone of a space and the join of two spaces are classical constructs in topology. We present their counterparts for pairs.

\subsection{Cone}\label{sec_cone}
\begin{definition}[Cone]
Let~$X$ be a topological space. The \textbf{cone} of~$X$ is
the quotient of~$X\times [0,1]$ under the equivalence relation~$(x,0)\sim (x',0)$. The \textbf{open cone} of~$X$ is the space~$\OC X=\C X\setminus X$, where~$X$ is embedded in~$\C{X}$ as~$X\times \{1\}$.
\end{definition}
The equivalence class of~$X\times\{0\}$ is called the \textbf{tip} of the cone. If~$X$ is embedded in~$Q$, then a realization of~$\C X$ in~$Q\cong[0,1]\times Q$ is given by
\begin{equation*}
\C X\cong\{(t,t x):t\in[0,1],x\in X\}.
\end{equation*}

\begin{definition}[Cone pair]
Let~$(X,A)$ be a pair. Its \textbf{cone pair} is the pair~$\C{X,A}=(\C X,X\cup\C A)$. Its \textbf{open cone pair} is the pair~$\OC{X,A}=(\OC{X},\OC{A})$.
\end{definition}

In particular,
\begin{align*}
\C{L,\emptyset}&=(\C{L},L)\\
\OC{L,\emptyset}&=(\OC{L},\emptyset).
\end{align*}

\paragraph{Special symbol.}
We introduce a symbol~$\unit$, which will be treated as a space but is purely formal. It enables to consider a singleton as a cone and is a unit for the join (see Section \ref{sec_join}). It can be thought as the boundary of the singleton, as in the definition of reduced homology. It is analogous to the empty simplicial complex~$\{\emptyset\}$, which behaves differently from the void complex~$\emptyset$. We define
\begin{equation}\label{eq_unit_cone}
\C{\unit}=\OC{\unit}=\{0\},
\end{equation}
where the point~$0$ is thought as the tip of a degenerate cone. We take the convention that~$\Sp_{-1}=\unit$.

\begin{example}
Let~$n\in\mathbb{N}$. One has
\begin{align*}
\C{\Sp_{n},\emptyset}&=(\BB_{n+1},\Sp_{n}),\\
\C{\BB_{n},\Sp_{n-1}}&=(\BB_{n+1},\Sp_{n}),
\end{align*}
with the tip at the center of~$\BB_{n+1}$ and in~$\Sp_n$ respectively. Note that the particular case~$n=0$ is consistent with the convention~$\Sp_{-1}=\unit$.
\end{example}

\subsection{Join}\label{sec_join}

\begin{definition}\label{def_join}
If~$X,Y$ are topological spaces, then their \textbf{join}~$X*Y$ is the quotient of~$X\times Y\times [0,1]$ under the equivalence relation~$(x,y,0)\sim (x,y',0)$ and~$(x,y,1)\sim (x',y,1)$.
\end{definition}

\begin{remark}[Special symbol]
In the literature, a convention is to let~$\emptyset$ be a unit for the join, i.e.~to define~$X*\emptyset$ as~$X$. Definition \ref{def_join} rather makes~$\emptyset$ an absorbing element, the unit element will be~$\unit$:
\begin{align*}
 X*\emptyset=\emptyset*X&=\emptyset,\\
X*\unit=\unit*X&=X.
 \end{align*}
\end{remark}

If~$X$ and~$Y$ are embedded in~$Q$, then~$X*Y$ can be realized as
\begin{equation*}
X*Y\cong\{(t,(1-t)x,ty):t\in[0,1],x\in X,y\in Y\}.
\end{equation*}
We also realize~$X*\unit\cong\{(0,x,0):x\in X\}$ and~$\unit*Y\cong\{(1,0,y):y\in Y\}$.

The cone and join constructs are intimately related:~$\C{X}=X*\{0\}$. It is proved in Brown \cite{Brown06} (Corollary 5.7.4, \S 5.7, p.~196) that a product of cones is a cone. More precisely, let~$X,Y$ be compact metrizable spaces. There is a homeomorphism
\begin{equation*}
\C{X*Y}\to \C{X}\times\C{Y}
\end{equation*}
which restricts to a homeomorphism
\begin{equation*}
X*Y\to (\C{X}\times Y)\cup (X\times \C{Y}).
\end{equation*}

We will use the following consequences:
\begin{proposition}[The product of cones]\label{prop_cone_product}
Let~$X,Y$ be compact metrizable spaces, or~$\unit$. One has
\begin{align*}
\C{X}\times\C{Y}&\cong\C{X*Y}\\
\OC{X}\times\OC{Y}&\cong\OC{X*Y}.
\end{align*}
\end{proposition}

\begin{proof}
The first equality is Corollary 5.7.4 in \cite{Brown06} for the general case. We need to check the particular cases when~$Y$ is~$\emptyset$ or~$\unit$. One has~$\C{X}\times \C{\emptyset}=\emptyset=\C{X*\emptyset}$ and~$\C{X}\times\C{\unit}=\C{X}\times \{0\}=\C{X*\unit}$, and the symmetric cases are similar. The second equality can be proved as follows for~$X,Y\neq\unit$,
\begin{align*}
\OC{X}\times\OC{Y} & =(\C{X}\setminus X)\times(\C{Y}\setminus Y)\\
 & =(\C{X}\times \C{Y})\setminus((\C{X}\times Y)\cup(X\times\C{Y}))\\
 & \cong\C{X*Y}\setminus (X*Y)\\
 & =\OC{X*Y},
\end{align*}
and when~$Y=\unit$,~$\OC{X}\times\OC{\unit}\cong\OC{X}=\OC{X*\unit}$.
\end{proof}

\subsubsection{Join of pairs}

The previous result extends to cone pairs. Note that in all the pairs~$(X,A)$ and~$(Y,B)$ that we consider,~$X$ and~$Y$ are never~$\emptyset$ or~$\unit$.
\begin{definition}
The \textbf{join} of two pairs~$(X,A)$ and~$(Y,B)$ is
\begin{equation*}
(X,A)*(Y,B)=(X*Y,X*B\cup A*Y).
\end{equation*}
\end{definition}


The cone pair is a particular case of the join of pairs:
\begin{align*}
(X,A)*(\{0\},\unit)&=(\C{X},X\cup\C{A})\\
&=\C{X,A},\\
(X,\emptyset)*(\{0\},\unit)&=(\C{X},X)\\
&=\C{X,\emptyset}.
\end{align*}

\begin{proposition}[Product of cone pairs]\label{prop_cone_pair_product}A product of (open) cone pairs is a (open) cone pair:
\begin{align*}
\C{X,A}\times\C{Y,B}&\cong\C{(X,A)*(Y,B)}\\
\OC{X,A}\times\OC{Y,B}&\cong\OC{(X,A)*(Y,B)}.
\end{align*}
\end{proposition}
\begin{proof}
The first component of~$\C{X,A}\times \C{Y,B}$ is~$\C{X}\times\C{Y}=\C{X*Y}$ by Proposition \ref{prop_cone_product}. Its second component is
\begin{align*}
&(\C{X}\times(Y\cup\C{B}))\cup((X\cup \C{A})\times \C{Y})\\
&=(\C{X}\times Y)\cup (X\times \C{Y})\cup (\C{X}\times\C{B})\cup (\C{A}\times\C{Y})\\
&\cong(X*Y)\cup \C{X*B}\cup\C{A*Y}\\
&=(X*Y)\cup \C{X*B\cup A*Y}.
\end{align*}

The first component of~$\OC{X,A}\times \OC{Y,B}$ is~$\OC{X}\times\OC{Y}\cong\OC{L}$ by Proposition \ref{prop_cone_product}. Its second component is
\begin{align*}
(\OC{X}\times\OC{B})\cup(\OC{A}\times\OC{Y})&\cong
\OC{X*B}\cup\OC{A*Y}\\
&=\OC{X*B\cup A*Y}.
\end{align*}
\end{proof}

\section{Hopf's extension and classification theorems}\label{sec_hopf}

We recall two classical results: Hopf's extension and classification theorem. They are usually expressed in terms of cohomology. The following statements using homology can be found in Hurewicz-Wallman \cite{HurWal41}. They are stated in terms of \v{C}ech homology for compact spaces in \cite{HurWal41}. We state them for compact ANRs using singular homology, which is equivalent to \v{C}ech homology for these spaces.

Hopf's extension theorem is Theorem VIII.1' in \cite{HurWal41} (\S VIII.6, p.~147).

\begin{theorem}[Hopf's extension theorem]\label{thm_hopf_extension}Let~$n\in\N$. 
Let~$(X,A)$ be a compact pair, where~$X$ and~$A$ are ANRs and~$\dim(X)\leq n+1$. For a continuous function~$f:A\to\Sp_n$, the following statements are equivalent:
\begin{itemize}
\item $f$ has a continuous extension~$F:X\to\Sp_n$,
\item $\ker i_*\subseteq \ker f_*$,
\end{itemize}
where~$i:A\to X$ is the inclusion map and
\begin{align*}
i_*:H_n(A;\RZ)&\to H_n(X;\RZ)\\
f_*:H_n(A;\RZ)&\to H_n(\Sp_n;\RZ)
\end{align*}
are the homomorphisms induced by~$i$ and~$f$.

The equivalence holds for~$n=0$ or~$1$, without any dimension assumption about~$X$.
\end{theorem}
For~$n=0$, Theorem \ref{thm_hopf_extension} also holds when considering homomorphisms between \emph{reduced} homology groups~$\widetilde{H}_0(\cdot;\RZ)$.

We present the argument for the particular case~$n=0$ or~$1$, which is not stated in \cite{HurWal41} but follows from the proof.

\begin{proof}
Assume that~$n\leq 1$. We only need to prove the result for CW-pairs~$(X,A)$, because a pair of compact ANRs is homotopy equivalent to a CW-pair. We assume that~$\ker i_*\subseteq\ker f_*$ and prove that~$f$ has an extension. From Hopf's extension theorem for spaces of dimension at most~$n+1$,~$f$ has an extension to the~$n+1$-skeleton of~$X$. For~$k\geq n+1$, every function from~$\Sp_k$ to~$\Sp_n$ is null-homotopic (because~$n=0$ or~$1$), so~$f$ can be inductively extended to every cell of~$X$.
\end{proof}
%
\begin{theorem}[Hopf's classification theorem]\label{thm_hopf_classification}
Let~$n\in\N$. Let~$(X,A)$ be a compact pair, where~$X$ and~$A$ are ANRs and~$\dim(X)\leq n$. For a continuous function~$f:(X,A)\to(\Sp_n,s)$, the following statements are equivalent:
\begin{itemize}
\item $f$ is null-homotopic (as a function of pairs),
\item $f_*:\widetilde{H}_n(X,A;\RZ)\to \widetilde{H}_n(\Sp_n,s;\RZ)$ is trivial.
\end{itemize}

The equivalence holds for~$n\leq 1$, without any dimension assumption about~$X$.
 \end{theorem}
When~$n=0$ and~$A=\emptyset$, and only in that case, one indeed needs to consider the homomorphism between reduced homology groups.
 \begin{proof}
This result is usually stated in terms of cohomology, and for single spaces rather than pairs (for instance, Theorem VIII.2, \S VIII.6, p.~149 in \cite{HurWal41}). It can be derived from Theorem \ref{thm_hopf_extension} as follows. 
Let~$g:X\cup\C{A}\to\Sp_n$ be defined as~$g=f$ on~$X$ and~$g=s$ on~$\C{A}$. A null-homotopy of~$f$ is a function~$G:\C{X}\to \Sp_n$ extending~$g$. The pair~$(\C{X},X\cup\C{A})$ satisfies all the conditions of Hopf's extension theorem, so~$g$ has extension if and only if~$\ker i_*\subseteq\ker g_*$. As~$\C{X}$ is contractible, its reduced homology groups are trivial so~$\ker i_*$ is the whole group~$\widetilde{H}_n(X\cup \C{A};\RZ)$; therefore,~$g$ has an extension iff~$g_*:\widetilde{H}_n(X\cup\C{A};\RZ)\to \widetilde{H}_n(\Sp_n;\RZ)$ is trivial. Up to natural isomorphisms,~$f_*$ is~$g_*$.
\end{proof}

\section{Coefficients}\label{sec_coef}
In Definition \ref{def_cycle} we used the coefficient group~$\RZ$ to detect relative cycles. One could equivalently use the groups~$\Z_k=\Z/k\Z$ (with~$k\geq 2$). The group~$\Z$ can also give information. Although the next result is probably folklore, we did not find a suitable reference and include a proof for completeness.

\begin{proposition}\label{prop_coefficients}
Let~$(X,A)$ be a compact pair and~$x\in X\setminus A$ be~$n$-Euclidean. The following statements are equivalent:
\begin{itemize}
\item The homomorphism~$H_n(X,A;\RZ)\to H_n(X,X\setminus\{x\};\RZ)$ is non-trivial,
\item The homomorphism~$H_n(X,A;\Z_k)\to H_n(X,X\setminus\{x\};\Z_k)$ is non-trivial for some~$k\geq 2$,
\end{itemize}
and they hold if the homomorphism~$H_n(X,A)\to H_n(X,X\setminus\{x\})$ is non-trivial.
\end{proposition}
\begin{proof}
We use the universal coefficient theorem for homology, and the fact that it is functorial in the space and in the coefficient group. For any abelian group~$G$, one has the following commuting diagram:
\begin{equation}\label{eq_univ}
\begin{tikzcd}
0 \arrow{r} & H_n(X,A)\otimes G \arrow{r}{g_1} \arrow{d}{f_1} & H_n(X,A;G) \arrow{r}{g_2} \arrow{d}{f_2} & \Tor(H_{n-1}(X,A),G) \arrow[r] \arrow{d}{f_3} & 0\\
0 \arrow[r]  & \underbrace{H_n(\Sp_n,s)\otimes G}_G \arrow{r}{h_1}  & \underbrace{H_n(\Sp_n,s;G)}_G \arrow{r}{h_2}  & \underbrace{\Tor(H_{n-1}(\Sp_n,s),G)}_0 \arrow[r] & 0
\end{tikzcd}
\end{equation}
We will write~$f^k_1$ for~$G=\Z_k$ and~$f'_1$ for~$G=\RZ$, and similarly for the other homomorphisms.

We first assume that~$f:H_n(X,A)\to H_n(\Sp_n,s)\cong \Z$ is non-trivial and show that~$f^k_2$ is non-trivial for some~$k\geq 2$. Let~$q\in H_n(\Sp_n,s)$ be a non-zero element of the image of~$f$. Let~$k\geq 2$ be an integer that does not divide~$q$, and take~$G=\Z_k$ in \eqref{eq_univ}, so that~$(q,1)$ is not zero in~$H_n(\Sp_n,s)\otimes G$. As~$(q,1)$ belongs to the image of~$f^k_1$,~$f^k_1$ is non-trivial. As~$h^k_1$ is injective,~$h^k_1\circ f^k_1=f^k_2\circ g^k_1$ is non-trivial, so~$f^k_2$ must be non-trivial.

We now assume that~$f_2^k$ is non-trivial for some~$k\geq 2$ and show that~$f'_2$ is non-trivial. The inclusion~$\Z_k\to\RZ$ sending~$[p]$ to~$[p/k]$ induces homomorphisms from~$H_n(Y,B;\Z_k)$ to~$H_n(Y,B;\RZ)$ which make the following diagram commute:
\begin{equation}\label{eq_univ3}
\begin{tikzcd}
H_n(X,A;\Z_k) \arrow{r}{i_2} \arrow{d}{f^k_2} & H_n(X,A;\RZ) \arrow{d}{f'_2}\\
H_n(\Sp_n,s;\Z_k) \arrow{r}{j_2}   & H_n(\Sp_n,s;\RZ)
\end{tikzcd}
\end{equation}
Note that~$j_2:\Z_k\to\RZ$ is the inclusion so it is injective. Therefore, the non-triviality of~$f^k_2$ implies the non-triviality of~$f'_2$.

Finally, we assume that~$f'_2$ is non-trivial and show that~$f^k_2$ is non-trivial for some~$k\geq 2$. Again, the inclusion~$\Z_k\to\RZ$ sending~$[m]$ to~$[m/k]$ induces the following commuting diagram:
\begin{equation}\label{eq_univ2}
\begin{tikzcd}
0 \arrow{r} & H_n(X,A)\otimes \Z_k \arrow{r}{g^k_1} \arrow{d}{i_1} & H_n(X,A;\Z_k) \arrow{r}{g^k_2} \arrow{d}{i_2} & \Tor(H_{n-1}(X,A),\Z_k) \arrow[r] \arrow{d}{i_3} & 0\\
0 \arrow[r]  & H_n(X,A)\otimes \RZ \arrow{r}{g'_1}   & H_n(X,A;\RZ) \arrow{r}{g'_2}  & \Tor(H_{n-1}(X,A),\RZ)  \arrow[r] & 0
\end{tikzcd}
\end{equation}

\begin{claim}\label{claim_sum}
For every~$c\in H_n(X,A;\RZ)$, there exist~$k\geq 2$,~$d\in H_n(X,A)\otimes \RZ$ and~$e\in H_n(X,A;\Z_k)$ such that~$c=g'_1(d)+i_2(e)$.
\end{claim}
\begin{proof}
Note that~$\Tor(H_{n-1}(X,A),\RZ)$ is the torsion subgroup of~$H_{n-1}(X,A)$. Therefore,~$g'_2(c)$ is a torsion element of~$H_{n-1}(X,A)$, so there exists~$k\geq 2$ such that~$g'_2(c)\in \im(i_3)$. As~$g^k_2$ is surjective,~$g'_2(c)=i_3\circ g^k_2(e)$ for some~$e\in H_n(X,A;\Z_k)$. One has~$g'_2\circ i_2(e)=i_3\circ g^k_2(e)=g'_2(c)$, so~$c-i_2(e)\in\ker(g'_2)=\im(g'_1)$, therefore there exists~$d\in H_n(X,A)\otimes\RZ$ such that~$c-i_2(e)=g'_1(d)$.
\end{proof}

Now assume that~$f'_2$ is non-trivial, and let~$c\in H_n(X,A;\RZ)$ have a non-zero image under~$f'_2$. One has~$c=g'_1(d)+i_2(e)$ as in Claim \ref{claim_sum}, so~$f'_2\circ g'_1(d)\neq 0$ or~$f'_2\circ i_2(e)\neq 0$. If~$f'_2\circ g'_1(d)\neq 0$, then~$f'_1$ is non-trivial by \eqref{eq_univ} so~$f$ is non-trivial and we can apply the first argument of the proof, implying that~$f^{r}_2$ is non-trivial for some~$r\geq 2$. If~$f'_2\circ i_2(e)\neq 0$, then~$f^k_2$ is non-trivial by \eqref{eq_univ3}.
\end{proof}

It can happen that the homomorphism is trivial with the coefficient group~$\Z$ but not with some~$\Z_k$. It happens for instance if~$X$ is an~$n$-manifold that is not orientable, with~$k=2$.

%
%
%
%
\end{document}